\documentclass[reqno]{amsart}

\usepackage{amssymb,latexsym,mathrsfs,amsmath}
\usepackage{amsthm}
\usepackage{graphicx}
\usepackage{diagbox}
\usepackage{mathrsfs}
\usepackage{tikz}
\usetikzlibrary{arrows}

\newtheorem{lemma}{Lemma}[section]
\newtheorem{proposition}[lemma]{Proposition}
\newtheorem{theorem}[lemma]{Theorem}
\newtheorem{corollary}[lemma]{Corollary}
\newtheorem{conjecture}[lemma]{Conjecture}

\theoremstyle{definition}
\newtheorem{remark}[lemma]{Remark}
\newtheorem{definition}[lemma]{Definition}
\newtheorem{example}[lemma]{Example}

\newtheorem{question}[lemma]{Question}

\DeclareMathOperator{\id}{id}

\DeclareMathOperator{\Kh}{Kh}
\DeclareMathOperator{\Hom}{Hom}
\DeclareMathOperator{\Sq}{Sq}

\DeclareMathOperator{\Ob}{Ob}
\DeclareMathOperator{\sq}{sq}

\newcommand{\Z}{\mathbb{Z}}

\newcommand{\F}{\mathbb{F}}

\newcommand{\Bu}{\mathscr{B}_\sigma}
\newcommand{\BN}{H_{\mathrm{BN}}}

\newcommand{\CBN}{C_{\mathrm{BN}}}

\newcommand{\CKh}{C_{\Kh}}

\newcommand{\cC}{\mathscr{C}}
\newcommand{\CC}{\mathfrak{C}}
\newcommand{\cCC}{{\CC\cC}}

\newcommand{\M}{\mathcal{M}}
\newcommand{\cn}{\underline{2}^n}
\newcommand{\op}{\mathrm{op}}
\newcommand{\Grd}{\mathfrak{GAb}}
\newcommand{\Kho}[1] {\Kh_{\mathrm{o}}^{#1}}

\newcommand{\crossing}[4] {
\draw[#4](#1,#2) -- (#1+#3,#2+#3);
\draw[#4](#1+#3, #2) -- (#1+0.6 * #3, #2+0.4 * #3);
\draw[#4](#1,#2+#3) -- (#1+0.4 * #3, #2 + 0.6 * #3);
}
\newcommand{\arroweast}[3] {
\draw[->] (#1,#2) -- (#1+#3, #2);
}
\newcommand{\smoothingup}[4]{
\draw[#4] (#1,#2) -- (#1+0.2 * #3, #2+0.2 * #3) to [out = 45, in = 315] (#1+0.2 * #3, #2+0.8 * #3) -- (#1, #2 + #3);
\draw[#4] (#1+#3,#2) -- (#1 + 0.8 * #3, #2 + 0.2 * #3) to [out = 135, in = 225] (#1 + 0.8 * #3, #2 + 0.8 * #3) -- (#1+#3,#2+#3);
}
\newcommand{\smoothinglr}[4]{
\draw[#4](#1,#2) -- (#1+0.2 * #3, #2+0.2 * #3) to [out = 45, in = 135] (#1 + 0.8 * #3, #2 + 0.2 * #3) -- (#1+#3,#2);
\draw[#4](#1,#2+#3) --  (#1+0.2 * #3, #2+0.8 * #3) to [out = 315, in = 225]  (#1+0.8 * #3, #2+0.8 * #3) -- (#1+#3,#2+#3);
}

\begin{document}
\parindent0em
\setlength\parskip{.1cm}
\thispagestyle{empty}
\title{Two Second Steenrod Squares for odd Khovanov homology}
\author[Dirk Sch\"utz]{Dirk Sch\"utz}
\address{Department of Mathematical Sciences\\ Durham University\\ United Kingdom}
\email{dirk.schuetz@durham.ac.uk}

\begin {abstract}
Recently, Sarkar--Scaduto--Stoffregen constructed a stable homotopy type for odd Khovanov homology, hence obtaining an action of the Steenrod algebra on Khovanov homology with $\Z/2\Z$ coefficients. Motivated by their construction we propose a way to compute the second Steenrod square. Our construction is not unique, but we can show it to be a link invariant which gives rise to a refinement of the Rasmussen $s$-invariant with $\Z/2\Z$ coefficients. We expect it to be related to the second Steenrod square arising from the Sarkar--Scaduto--Stoffregen construction.
\end {abstract}

\maketitle

\section{Introduction}
In \cite{MR3230817}, Lipshitz and Sarkar constructed a stable homotopy type for Khovanov homology based on the framed flow categories of \cite{MR1362832}. One advantage of this construction is that they were able to make explicit calculations of the second Steenrod square \cite{MR3252965}, thus showing that their construction carries more information than the homology groups themselves.

But structural questions such as the behaviour of this homotopy type under split or connected sums, or mirroring are not easily derived from this construction. However, in joint work with Lawson \cite{MR4153651} a new construction based on functors from a cube category $\cn$ to the Burnside category $\mathcal{B}$ (see Section \ref{ss:cube} and \ref{ss:burn} for the definitions) was developed, which was shown to be equivalent to the original, and could be used to prove the structural conjectures of \cite{MR3230817}.

This latter construction was further generalized in \cite{MR4078823} by Sarkar, Scaduto and Stoffregen by allowing functors to a {\em signed} Burnside category $\mathcal{B}_\sigma$, which then resulted in a stable homotopy type for odd Khovanov homology via a realization construction.

It is worth pointing out that given a functor $F\colon \cn\to\mathcal{B}$ a framed flow category $\cC_F$ is constructed in \cite{MR4153651} which gives rise to the original construction of \cite{MR3230817}. It is not obvious how to extend this construction to functors $F_\sigma\colon\cn\to\mathcal{B}_\sigma$, possibly explaining the lack of computations for the odd homotopy type. Note that such functors give rise to well defined cohomology groups $H^\ast(F_\sigma;G)$ for any abelian group $G$.

While we cannot construct a framed flow category $\cC_{F_\sigma}$, we do succeed in constructing a weaker object, a framed $1$-flow category. These objects carry enough information to define a second Steenrod square. Our construction involves various choices, and it turns out that we cannot make it independent of all of them. We can denote our first result as follows.

\begin{proposition}\label{prop:firstprop}
Let $F_\sigma\colon \cn\to\mathcal{B}_\sigma$ be a strictly unitary $2$-functor and $\varepsilon\in \Z/2\Z$. Then there exist linear maps
\[
\Sq^2_\varepsilon\colon H^i(F_\sigma;\Z/2\Z) \to H^{i+2}(F_\sigma;\Z/2\Z),
\]
which are non-zero in general, and can depend on $\varepsilon$.
\end{proposition}

Given a link diagram $D$ with $n$ crossings, Sarkar--Scaduto--Stoffregen construct a functor $F_o^j\colon \cn\to \mathcal{B}_\sigma$ whose cohomology groups agree with odd Khovanov homology of the link in $q$-degree $j$. Applying Proposition \ref{prop:firstprop} gives rise to two second Steenrod squares on odd Khovanov homology. These operations do not depend on the various choices, and are in fact link invariants.

\begin{theorem}\label{thm:firstthm}
Let $L$ be a link and $\varepsilon\in \Z/2\Z$. Then there exist link invariants
\[
\Sq^2_\varepsilon\colon \Kho{i,j}(L;\Z/2\Z) \to \Kho{i+2,j}(L;\Z/2\Z)
\]
which are non-trivial in general.
\begin{enumerate}
\item The only knot with up to $8$ crossings for which $\Sq^2_0$ is non-trivial is the torus knot $T(3,4)$.
\item The only knot with up to $8$ crossings for which $\Sq^2_1$ is non-trivial is the mirror of $T(3,4)$.
\end{enumerate}
\end{theorem}

The subscript `o', indicating odd Khovanov homology, can be dropped, since over $\Z/2\Z$ coefficients the odd theory agrees with Khovanov's original homology. In particular, from Lipshitz and Sarkar's original stable homotopy type we get another second Steenrod square $\Sq^2$ which on all knots with up to $8$ crossings is only non-trivial on $T(3,4)$ and its mirror. We note that $T(3,4)$ and its mirror are the only possible knots with up to $8$ crossings for which a second Steenrod square can be non-trivial. Calculations show that for knots which support Khovanov homology that can have a non-trivial second Steenrod square, non-triviality is quite common.

From these first calculations we see that the three second Steenrod squares are all different. Maybe more surprising is that the odd versions fail to satisfy a Spanier--Whitehead duality for a link and its mirror that is present in the Lipshitz--Sarkar stable homotopy type. Another immediate consequence of Theorem \ref{thm:firstthm} is that our second Steenrod squares do not split over reduced odd Khovanov homology.

There is also the second Steenrod square coming from the stable homotopy type of \cite{MR4078823}, and there is the natural question what its relation is to our Steenrod squares. The behaviour of this Steenrod square is conjectured to split over the reduced theory \cite[Rm.5.7]{MR4078823}, and to satisfy the Spanier--Whitehead duality between a link and its mirror \cite[Qn.5]{MR4078823}. Clearly at most one of our Steenrod squares can agree with theirs, but understanding the relation more closely would also shed more information on their general construction. We also note that their constructions lead to potentially more than one stable homotopy type for the odd theory, compare \cite[Rm.5.6]{MR4078823}. Since we use the same functor from the cube category to the signed Burnside category as \cite{MR4078823}, we believe that there should be a relation.

One application of the stable homotopy type is to get refinements of the Rasmussen invariant $s_\F(K)$ \cite{MR2729272} for $\F$ a field. A construction for Khovanov homology was given in \cite{MR3189434}, and a version for odd Khovanov homology was established in \cite[\S 5.6]{MR4078823}. The definition carries over to our situation, and we obtain that the resulting invariant $s_\pm^{\Sq^2_\varepsilon}(K)$ is a concordance invariant. Our calculations show that this invariant can differ from the Rasmussen invariant, similar to the results in \cite{MR3189434}.

\begin{theorem}\label{thm:secthm}
Let $K$ be the mirror of the knot $9_{42}$. Then
\[
s_+^{\Sq^2_0}(K) = 2,
\]
while the Rasmussen invariant of this knot is $0$.
\end{theorem}

Computer calculations show that $s_\pm^{\Sq^2_0}(K) = s_\pm^{\Sq^2}(K)$ and $s_\pm^{\Sq^2_1}(K) = s_\F(K)$  for all knots with up to $13$ crossings, but this is no longer true for knots with $15$ crossings or more. However, computation times get very long for knots with $14$ or more crossings and we have not systematically examined them.

{\bf Acknowledgements.} The author would like to thank Sucharit Sarkar for valuable comments improving the presentation.

\section{$1$-Flow Categories}
Instead of flow categories, see \cite{MR1362832,MR3230817}, we are going to use a simplification. The basic $1$-flow category we will need is build over a cube. We now give several constructions that will be needed.

\subsection{The cube category}\label{ss:cube}
We begin with the cube category as defined in \cite{MR4153651}.

For $n$ a positive integer let $\cn$ be the category with object set $\Ob(\cn) = \{0,1\}^n$. To describe the morphisms, give $\{0,1\}^n$ the partial order $u=(u_1,\ldots,u_n) \geq v = (v_1,\ldots,v_n)$ if and only if $u_i\geq v_i$ for all $i=1,\ldots,n$. The morphism sets are then defined by
\[
\Hom_{\cn}(u,v) = \left\{ \begin{array}{cc} \{\phi_{u,v}\} & \mbox{if }u\geq v \\
\emptyset & \mbox{otherwise}
\end{array}\right.
\]
Here $\phi_{u,v}$ symbolizes the unique morphism for $u\geq v$.

We call $\cn$ the {\em $n$-dimensional cube category}. It carries a grading on objects given by
\[
|u| = \sum_{i=1}^n u_i.
\]
For $k$ a non-negative integer we write $u\geq_k v$ if $u\geq v$ and $|u|-|v|=k$. Given two objects $u,w$ with $u\geq_k w$ let $C_{u,w}$ be the full sub-category of $\cn$ with
\[
\Ob(C_{u,w}) = \{v\in \Ob(\cn)\mid u\geq v\geq w\},
\]
which we call a {\em $k$-dimensional sub-cube} of $\cn$.

We now get a cochain complex $(C^\ast(\cn;\Z/2\Z),\delta)$, where $C^k(\cn;\Z/2\Z)$ is the dual of the  $\Z/2\Z$-vector space generated by the $k$-dimensional sub-cubes of $\cn$, and
\[
\delta(C^\ast_{u,w}) = \sum [C_{r,t}:C_{u,w}] C^\ast_{r,t},
\]
where the sum is over the $(k+1)$-dimensional sub-cubes $C_{r,t}$, and $[C_{r,t}:C_{u,w}] =1$ if $C_{u,w}$ is a sub-category of $C_{r,t}$, and $0$ otherwise. Note that $C^\ast_{u,v}$ is the dual of $C_{u,v}$.

This cochain complex can be identified with the CW-cochain complex of the $n$-dimensional cube $[0,1]^n$ with its usual CW-structure. In particular, the cohomology groups for $k\geq 1$ all vanish.

\begin{definition}
A {\em sign assignment} for $\cn$ is a cochain $s\in C^1(\cn;\Z/2\Z)$ with the property that $\delta s(C_{u,w}) = 1$ for every $2$-dimensional sub-cube $C_{u,w}$, where $\delta$ is the coboundary. The {\em standard sign assignment} $s^\ast\in C^1(\cn;\Z/2\Z)$ is defined by
\[
s^\ast(C_{u,v}) = \sum_{j=1}^{i-1} v_j \bmod 2,
\]
where $i\in \{1,\ldots,n\}$ is the unique integer with $v_i<u_i$.
\end{definition}

\subsection{Framed 1-flow categories}
These were introduced in \cite{MR4165986} as a simplification of framed flow categories, which still carry enough information to define a second Steenrod square.

\begin{definition}\label{def_1flow}
A \em $1$-flow category $\cC$ \em consists of a finite set $\Ob(\cC)$, a function ${|\cdot|}\colon \Ob(\cC)\to \Z$ called the \em grading\em, and for each pair $a,b\in \Ob(\cC)$ with $|a|-|b|=1$ or $2$ a \em moduli space \em $\M(a,b)$ which is a compact manifold of dimension $|a|-|b|-1$, satisfying the following
\begin{itemize}
 \item {\bf Boundary Condition:}
If $a,c\in \Ob(\cC)$ with $|a|-|c|=2$, then the boundary of $\M(a,c)$ is given by
\[
 \partial \M(a,c)=\coprod_{b\in \Ob(\cC),|b|=|a|-1} \M(b,c)\times \M(a,b)
\]
\end{itemize}

Given a $1$-flow category and $a,d\in \Ob(\cC)$ with $|a|-|d|=3$, rather than defining a moduli space $\M(a,d)$, we only define its boundary as
\begin{align*}
 \partial \M(a,d)= & \coprod_{b\in \Ob(\cC),|b|=|a|-1} \M(b,d)\times \M(a,b) \\
 & \cup \coprod_{c\in \Ob(\cC),|c|=|a|-2}\M(c,d)\times \M(a,c)
\end{align*}
Notice that the two disjoint unions have a common subset, which is
\[
 \coprod_{\stackrel{(b,c)\in \Ob(\cC)\times \Ob(\cC)}{|b|=|c|+1=|d|+2}} \M(c,d)\times \M(b,c)\times \M(a,b)
\]
It is easy to see that $\partial\M(a,d)$ is a disjoint union of components which are homeomorphic to circles.
\end{definition}

We will use a subscript $\M_\cC$ if we want to emphasize the dependence of the moduli space on $\cC$ in the presence of other 1-flow categories.

\begin{definition}
 Let $\cC$ be a $1$-flow category. A \em sign assignment $s$ \em for $\cC$ is an assignment $s(P)\in \Z/2\Z$ for every point $P$ in a $0$-dimensional moduli space $\M(a,b)$ with the property that if $(P_1,Q_1)\in \M(b_1,c)\times \M(a,b_1)$ and $(P_2,Q_2)\in \M(b_2,c)\times \M(a,b_2)$ are the boundary of an interval component in $\M(a,c)$, then 
\[
 s(P_1)+s(Q_1)+s(P_2)+s(Q_2)=1.
\]
For objects $a,b\in \Ob(\cC)$ with $|a|=|b|+1$ we define $[a:b]\in \Z$ as
\[
 [a:b] = \sum_{A\in \M(a,b)} (-1)^{s(A)}.
\]
A \em pre-framing $f$ \em of $\cC$ is an assignment $f(C)\in \Z/2\Z$ for every component $C\subset \M(a,c)$ of a $1$-dimensional moduli space.
\end{definition}
Again we may write $[a:b]_\cC$ in the presence of other 1-flow categories.

A $1$-flow category together with a sign assignment $s$ gives rise to a cochain complex $C^\ast(\cC)$ where each $C^k(\cC)$ is the free abelian group generated by objects $a$ with $|a|=k$, and the coboundary is given by
\[
 \delta(b)=\sum_{a\in \Ob(\cC), |a|=|b|+1}[a:b]a.
\]

We can think of a sign assignment as a framing of the $0$-dimensional moduli spaces, with $0$ corresponding to a positive framing, and $1$ to a negative framing. If we embed a $1$-dimensional closed manifold into an at least $4$-dimensional Euclidean space, each component can have two different framings up to framed cobordism. Similarly, if we embed an interval into a half-Euclidean space of high enough dimension with a fixed framing on the boundary, we can get two different framings up to framed cobordism relative to the boundary.
We thus think of a pre-framing as a choice of framing of the components of the $1$-dimensional moduli spaces. In order for this to be useful, we need to add a condition.

Let $C$ be a component in $\partial \M(a,d)$. If $C$ is a circle, which is either of the form $\{P\}\times S^1\subset \M(c,d)\times \M(a,c)$ or $S^1\times \{Q\}\subset \M(b,d)\times \M(a,b)$, we define 
\[
\tilde{f}(C)=f(S^1)\in \Z/2\Z,
\]
 where $f(S^1)$ is the framing value of the circle. If $C$ is a union of intervals $\{P\}\times J\subset \M(c,d)\times \M(a,c)$ and $I \times \{Q\}\subset \M(b,d)\times \M(a,b)$, we define
\[
 \tilde{f}(C) = \sum_{I \times \{Q\}}f(I)+ \sum_{\{P\}\times J} (1+s(P)+f(J))\in \Z/2\Z
\]
where the first sum is over all intervals of the form $I\times\{Q\}\subset C$ and the second sum over all intervals of the form $\{P\}\times J$.

\begin{definition}
 Let $\cC$ be a $1$-flow category, $s$ a sign assignment and $f$ a pre-framing of $\cC$. Then $(\cC,s,f)$ is called a \em framed $1$-flow category\em, if the following \em compatibility condition \em is satisfied for $f$.
\begin{itemize}
 \item Let $a,d\in \Ob(\cC)$ satisfy $|a|=|d|+3$. Then
\[
 \sum_{C} (1+\tilde{f}(C)) = 0 \in \Z/2\Z
\]
where the sum is taken over all components in $\partial \M(a,d)$.
\end{itemize}
If $f$ satisfies the condition with respect to $s$, we call $f$ a \em framing of $\cC$\em, or a {\em frame assignment for $\cC$}.
\end{definition}

\begin{remark}
The sign and frame assignments can be used to turn the $0$- and $1$-dimensional moduli spaces into framed manifolds. These framings then induce a framing on $\partial \M(a,d)$, which is topologically the disjoint union of finitely many circles. The compatibility condition ensures this framing is framed null-cobordant. For cube 1-flow categories this is implicit in the proof of \cite[Lm.3.5]{MR3252965}. There, a framed circle $S^1$ represents an element in $H_1(SO)\cong \Z/2\Z$ such that the trivially framed circle represents the non-trivial element of $H_1(SO)$. As a result, for a trivially framed circle component $C$ we need to have that $1+\tilde{f}(C)$ contributes $0$ to the compatibility condition. That is, we need $f(S^1) = 1$ for a trivially framed circle $S^1$ in the frame assignment. The combinatorial Whitney trick \cite[\S 4.2]{MR4165986} also requires $f(S^1) = 1$ for trivially framed circles, but $\tilde{f}(C)$ was defined as $1+f(S^1)$ there, leading to an inconsistency in \cite{MR4165986}. By setting $\tilde{f}(C) = f(S^1)$ and using $f(S^1) = 1$ for a trivially framed circle we get that the combinatorial Whitney trick works as intended.
\end{remark}

\begin{example}
For $n$ a positive integer define the {\em $n$-dimensional cube $1$-flow category} $\cC(n)$ as follows.
The object set is given by $\{0,1\}^n$ with grading
\[
|(u_1,\ldots, u_n)| = \sum_{i=1}^n u_i.
\]
Notice that the object set and grading agree with those for $\cn$. We can now define the $0$-dimensional moduli spaces as
\[
\M_{\cC(n)}(u,v) = \left\{ \begin{array}{cc} \{C_{u,v}\} & \mbox{if }u\geq_1 v \\
\emptyset & \mbox{otherwise}
\end{array}
\right.
\]
If $u\geq_2 w$, there exist exactly two objects $v_1,v_2$ with $u\geq_1 v_i\geq_1 w$, and we define the $1$-dimensional moduli space $\M_{\cC(n)}(u,w)$ to be a compact interval between the points $(C_{v_1,w},C_{u,v_1})$ and $(C_{v_2,w},C_{u,v_2})$. All other $1$-dimensional moduli spaces are empty.

\end{example} 

Any sign assignment for $\cn$ also defines a sign assignment for $\cC(n)$. In particular, we have the standard sign assignment $s^\ast$. Notice that a pre-framing for $\cC(n)$ can be interpreted as a cochain $f\in C^2(\cn;\Z/2\Z)$.

\begin{definition}
The {\em standard frame assignment} $f^\ast\in C^2(\cn;\Z/2\Z)$ is defined by
\[
f^\ast(C_{u,w}) = (w_1+\cdots + w_{i-1}) \cdot (w_{i+1}+\cdots+w_{j-1}),
\] 
where $i<j$ are the unique integers with $w_i < u_i$ and $w_j < u_j$.
\end{definition}

\begin{lemma}
The standard frame assignment $f^\ast$ and the standard sign assignment $s^\ast$ turn $(\cC(n), s^\ast, f^\ast)$ into a framed $1$-flow category.
\end{lemma}

\begin{proof}
We need to check that the compatibility condition is satisfied for any objects $u,x$ with $|u| = |x|+3$.  Given a $2$-dimensional sub-cube $C_{u,w}$ and integers $i<j$ with $w_i<u_i$ and $w_j<u_j$ write
\[
a = w_1+\cdots + w_{i-1} \mbox{ and } b = w_1+\cdots + w_{j-1},
\]
with $a,b\in \Z/2\Z$. Since $w_i = 0$ we get $f^\ast(C_{u,w}) = a(a+b)$.

Now let $C_{u,x}$ be a $3$-dimensional sub-cube of $\cn$, which spans a sub-$1$-flow category that we can visualize as follows (with letters indicating the sign assignment of each edge).
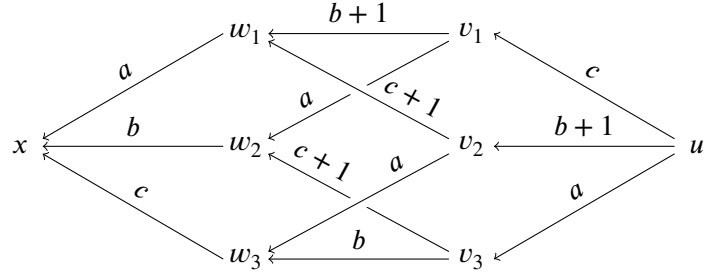
\[
\begin{tikzpicture}
\node at (0,0) {$x$};
\node at (3,1.5) {$w_1$};
\node at (3,0) {$w_2$};
\node at (3,-1.5) {$w_3$};
\node at (6,1.5) {$v_1$};
\node at (6,0) {$v_2$};
\node at (6,-1.5) {$v_3$};
\node at (9,0) {$u$};
\draw[<-] (0.3, 0.1) -- node [above, sloped] {$a$} (2.7,1.5);
\draw[<-] (0.3, 0) -- node [above] {$b$} (2.7, 0);
\draw[<-] (0.3, -0.1) -- node [above, sloped] {$c$} (2.7, -1.5);
\draw[<-] (3.3, 0.1) -- node [above, sloped, near start] {$a$} (5.7, 1.4);
\draw[<-] (3.3, -0.1) -- node [above, sloped, near start] {$c+1$} (5.7, -1.4);
\draw[-, white, line width = 6pt] (3.3, 1.4) -- (5.7, 0.1);
\draw[-, white, line width = 6pt] (3.3, -1.4) -- (5.7, -0.1);
\draw[<-] (3.3, 1.5) -- node [above] {$b+1$} (5.7, 1.5);
\draw[<-] (3.3, 1.4) -- node [above, near end, sloped] {$c+1$} (5.7, 0.1);
\draw[<-] (3.3, -1.4) -- node [above, near end, sloped] {$a$} (5.7, -0.1);
\draw[<-] (3.3, -1.5) -- node [above] {$b$} (5.7, -1.5);
\draw[<-] (6.3, 1.5) -- node [above, sloped] {$c$} (8.7, 0.1);
\draw[<-] (6.3, 0) -- node [above] {$b+1$} (8.7, 0);
\draw[<-] (6.3, -1.5) -- node [above, sloped] {$a$} (8.7, -0.1);
\end{tikzpicture}
\]
Notice that since we use the standard sign assignment, all edges moving from the lower left to the upper right are labelled with $a$, while all the edges moving from the upper left to the lower right are labelled with $c$, with an additional $+1$ or $+2$ depending on their level of grading. The horizontal edges only get a $+1$ if the first edge (labelled with $a$) preceded it.

We get that $\partial \M_{\cC(n)}(u,x)$ is a hexagon $C$ with three intervals $I\times \{Q\}$ of the form $\M_{\cC(n)}(v_m,x)\times \M_{\cC(n)}(u, v_m)$, together with three intervals $\{P\}\times J$ of the form $\M_{\cC(n)}(w_m,x)\times \M_{\cC(n)}(u, w_m)$.

The contributions to $\tilde{f}^\ast(C)$ are
\[
a(a+b) + a(a+c) + b(b+c) \in \Z/2\Z
\]
from the intervals $I\times \{Q\}$, and
\[
1+a+b+c+ (b+1)(b+c) + a(a+c+1) + a(a+b)\in \Z/2\Z
\]
from the intervals $\{P\}\times J$. Expanding out and adding the two lines shows that $\tilde{f}^\ast(C) =1$, so the compatibility condition is satisfied.
\end{proof}

We will need to look at sign and frame assignments which are not the standard ones. To begin with, assume that $s$ and $s'$ are sign assignments for $\cC(n)$. Also, assume that $f$ is a pre-framing so that $(\cC(n), s,f)$ is framed. We want to get a pre-framing $f'$ so that $(\cC(n), s', f')$ is framed.

Given a $2$-dimensional sub-cube $C_{u,w}$, there exist two objects $v_1,v_2$ with $u\geq_1 v_i \geq_1 w$. We can visualize the cube as
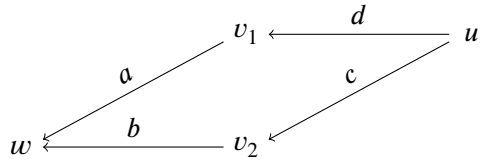
\[
\begin{tikzpicture}
\node at (0,0) {$w$};
\node at (3,1.5) {$v_1$};
\node at (3,0) {$v_2$};
\node at (6,1.5) {$u$};
\draw[<-] (0.3,0.1) -- node [above, sloped] {$a$} (2.7, 1.4);
\draw[<-] (0.3, 0) -- node [above] {$b$} (2.7, 0);
\draw[<-] (3.3, 1.5) -- node [above] {$d$} (5.7, 1.5);
\draw[<-] (3.3, 0.1) -- node [above, sloped] {$c$} (5.7, 1.4);
\end{tikzpicture}
\]
Here $a,b,c,d\in \Z/2\Z$ are the values of $s$ on the respective edges. In particular, $a+b+c+d = 1$. Let us denote the values of $s'$ on these edges as $a', b', c', d'\in \Z/2\Z$. Since $a'+b'+c'+d' = 1$, we have an even number of differences for $s$ and $s'$ on the four edges in $C_{u,w}$. With $\delta,\varepsilon\in \Z/2\Z$ we now define a new pre-framing $f'_{\delta,\varepsilon}$ by
\begin{equation}\label{eq:framechange}
f'_{\delta,\varepsilon}(C_{u,w}) = 
\left\{
\begin{array}{cc}
f(C_{u,w}) & \mbox{if } a=a', b=b', c=c', d=d' \\[0.2cm]
f(C_{u,w}) + 1 & \mbox{if } a\not=a', b\not=b', c=c', d=d' \\[0.2cm]
f(C_{u,w}) + c + d & \mbox{if } a=a', b=b', c\not=c', d\not=d' \\[0.2cm]
f(C_{u,w}) + \delta + a & \mbox{if } a\not=a', b=b', c=c', d\not=d' \\[0.2cm]
f(C_{u,w}) + \delta + b & \mbox{if } a=a', b\not=b', c\not=c', d=d' \\[0.2cm]
f(C_{u,w}) + \delta + \varepsilon + b & \mbox{if } a\not=a', b=b', c\not=c', d=d' \\[0.2cm]
f(C_{u,w}) + \delta + \varepsilon + a & \mbox{if } a=a', b\not=b', c=c', d\not=d' \\[0.2cm]
f(C_{u,w}) + a + b & \mbox{if } a\not=a', b\not=b', c\not=c', d\not=d' 
\end{array}
\right.
\end{equation}

Before we show that this gives a framing for $(\cC(n),s')$, let us take a closer look at (\ref{eq:framechange}). By interpreting $s$ and $s'$ as cochains in $C^1(\cn;\Z/2\Z)$ we see that $s+s'$ is a cocycle, and hence a coboundary. Furthermore, if $x$ is an object in $\cC(n)$, the coboundary $\delta(x)$  is non-zero exactly at every edge containing $x$.

So the various cases in (\ref{eq:framechange}) can be thought of as coming from changes in the sign assignment by $\delta(x)$ with $x$ a linear combination of $w, v_1, v_2, u$. If $s+s' = \delta(w)$, we get $a\not= a', b\not= b', c=c', d=d'$, and we are in the case of the second line of (\ref{eq:framechange}). Similarly, the lines three to five come from $\delta(u), \delta(v_1), \delta(v_2)$ respectively. The last line can be obtained from $\delta(w)+\delta(u)$, or from $\delta(v_1)+\delta(v_2)$, but both lead to the same value for $f'$.

For lines six and seven there is an ambiguity. Notice that $a\not=a', b = b', c\not= c', d=d'$ can be obtained from $\delta(w)+\delta(v_2)$. If we first change via $\delta(w)$, the framing value changes by $1$, and if we then change via $\delta(v_2)$ the framing changes by another $\delta+b+1$. Notice that the edge $C_{v_2,w}$ had sign $b+1$ after the change via $\delta(w)$. In total, the framing value changed by $\delta+b$, which is what is recorded in line six with $\varepsilon = 0$. However, if we first change via $\delta(v_2)$, the framing value changes by $\delta+b$, and after the change via $\delta(w)$ we add another $1$ to get a total change by $\delta + 1+b$, which is line six with $\varepsilon = 1$.

Similarly, we can get $a\not=a', b = b', c\not= c', d=d'$ from $\delta(v_1)+\delta(u)$, and we have a similar behavior in that first changing via $\delta(v_1)$ and then via $\delta(u)$ leads to a frame change by $\delta+b$, while doing it in the other order leads to a frame change by $\delta + 1 + b$.

We resolve this ambiguity by noticing that changes via $\delta(x)$ for various objects $x$ in {\em increasing} homological grading lead to the formula in (\ref{eq:framechange}) with $\varepsilon = 0$, while changing in decreasing order leads to the formula with $\varepsilon = 1$. With this convention, lines six and seven are also implied by lines two to five. Note that line one can be achieved from $\delta(w)+\delta(v_1)+\delta(v_2)+\delta(u)$, but if we make four changes in increasing homological order we also get no total change in framing.

\begin{lemma}\label{lm:frame4signchange}
If $(\cC(n), s, f)$ is framed and $\delta, \varepsilon \in \Z/2\Z$, then $(\cC(n), s', f'_{\delta,\varepsilon})$ with $f'_{\delta, \varepsilon}$ given by (\ref{eq:framechange}) is framed.
\end{lemma}

\begin{proof}
We will simply write $f'$ for $f'_{\delta,\varepsilon}$. After the above discussion it is now enough to assume that $s+s' = \delta(x)$ for exactly one object $x$. Consider a $3$-dimensional sub-cube $C_{u,x}$ whose sub-1-flow category is given by
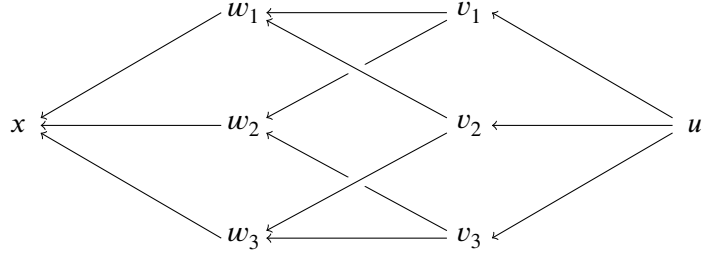
\[
\begin{tikzpicture}
\node at (0,0) {$x$};
\node at (3,1.5) {$w_1$};
\node at (3,0) {$w_2$};
\node at (3,-1.5) {$w_3$};
\node at (6,1.5) {$v_1$};
\node at (6,0) {$v_2$};
\node at (6,-1.5) {$v_3$};
\node at (9,0) {$u$};
\draw[<-] (0.3, 0.1) --  (2.7,1.5);
\draw[<-] (0.3, 0) --  (2.7, 0);
\draw[<-] (0.3, -0.1) -- (2.7, -1.5);
\draw[<-] (3.3, 0.1) --  (5.7, 1.4);
\draw[<-] (3.3, -0.1) --  (5.7, -1.4);
\draw[-, white, line width = 6pt] (3.3, 1.4) -- (5.7, 0.1);
\draw[-, white, line width = 6pt] (3.3, -1.4) -- (5.7, -0.1);
\draw[<-] (3.3, 1.5) --  (5.7, 1.5);
\draw[<-] (3.3, 1.4) -- (5.7, 0.1);
\draw[<-] (3.3, -1.4) -- (5.7, -0.1);
\draw[<-] (3.3, -1.5) --  (5.7, -1.5);
\draw[<-] (6.3, 1.5) --  (8.7, 0.1);
\draw[<-] (6.3, 0) --  (8.7, 0);
\draw[<-] (6.3, -1.5) --  (8.7, -0.1);
\end{tikzpicture}
\]
Since the compatibility condition is satisfied for the corresponding hexagon, we have
\begin{multline*}
0 = s(C_{w_1,x}) + s(C_{w_2,x})+ s(C_{w_3,x}) \\ +f(C_{v_1,x})+f(C_{v_2,x}) + f(C_{v_3,x}) + f(C_{u,w_1})+ f(C_{u,w_2})+f(C_{u,w_3}).
\end{multline*}
We now need to consider the eight cases arising by changing the sign assignment via the coboundary of an object which is part of the cube, and check that the compatibility condition is still satisfied. There are four main cases, with two having very similar sub-cases.

If we change the sign assignment by $\delta(x)$, we get
\[
s'(C_{w_1,x}) + s'(C_{w_2,x})+ s'(C_{w_3,x}) = s(C_{w_1,x}) +1 + s(C_{w_2,x})+ 1 + s(C_{w_3,x}) +1,
\]
and
\[
f'(C_{v_1,x})+f'(C_{v_2,x}) + f'(C_{v_3,x}) = f(C_{v_1,x})+1+ f(C_{v_2,x}) +1 + f(C_{v_3,x}) +1.
\]
The other $2$-dimensional sub-cubes are not changing frame, so the total change is $0$, and the compatibility condition is still satisfied.

If we change the sign assignment by $\delta(w_1)$, we get
\[
s'(C_{w_1,x}) + s'(C_{w_2,x})+ s'(C_{w_3,x}) = s(C_{w_1,x}) +1 + s(C_{w_2,x})+ s(C_{w_3,x}) ,
\]
while the three sub-cubes that change their framing do this as follows.
\begin{multline*}
f'(C_{v_1,x})+f'(C_{v_2,x}) + f'(C_{u, w_1}) = \\
f(C_{v_1,x})+\delta+s(C_{w_1,x})+f(C_{v_2,x}) +\delta+ s(C_{w_1,x}) + f(C_{u, w_1})+1.
\end{multline*}
Again, the changes cancel each other out and we still have the compatibility condition.

Changing via $\delta(w_2)$ and $\delta(w_3)$ is practically the same.

If we change via $\delta(v_1)$, the signs of the first three edges no longer change, and we only need to look at what happens to the three involved $2$-dimensional sub-cubes. Then we get
\begin{multline*}
f'(C_{v_1,x}) + f'(C_{u,w_1}) + f'(C_{u, w_2}) = f(C_{v_1,x}) + s(C_{v_1, w_1})+s(C_{v_1, w_2}) + \\ f(C_{u,w_1}) +\delta + s(C_{v_1, w_1}) + f(C_{u, w_2}) +\delta + s(C_{v_1, w_2}).
\end{multline*}
In particular, the contribution does not change. The same happens if we change via $\delta(v_2)$ and $\delta(v_3)$.

Finally, consider the change via $\delta(u)$. Again, the signs of the first three edges do not change, and we only need to check that the contribution of the three affected $2$-dimensional sub-cubes does not change. Here we have
\begin{multline*}
 f'(C_{u,w_1}) + f'(C_{u, w_2}) + f'(C_{u,w_3}) = f(C_{u,w_1}) + s(C_{u,v_1}) + s(C_{u,v_2}) + \\ 
 f(C_{u, w_2}) + s(C_{u,v_1}) + s(C_{u,v_3}) + f(C_{u,w_3}) + s(C_{u,v_2}) + s(C_{u,v_3}).
\end{multline*}
Again, all changes cancel each other, and the compatibility condition remains satisfied.
\end{proof}

If we interpret a pre-framing $f$ on $\cC(n)$ as a cochain $f\in C^2(\cn;\Z/2\Z)$, the compatibility condition can be expressed as 
\[
\delta(f)(C_{u,x}) = s(C_{w_1,x})+s(C_{w_2,x})+s(C_{w_3,x}),
\]
for all $3$-dimensional sub-cubes $C_{u,x}$, with $C_{w_i,x}$ the obvious edges, and $s$ a sign-assignment.

In particular, if $f,f'$ are framings for $(\cC(n),s)$, we get that $f+f'$ is a cocycle, hence a coboundary. Let us record this for future reference.

\begin{lemma}\label{lm:changefrm}
Let $(\cC(n),s,f)$ and $(\cC(n),s,f')$ be framed. Then there exists a cochain $g\in C^1(\cn;\Z/2\Z)$ with $f+f' = \delta(g)$. Furthermore, adding a coboundary to a frame assignment gives another frame assignment. \hfill\qed 
\end{lemma}

\subsection{The second Steenrod Square of a framed 1-flow category}

Framed $1$-flow categories carry enough information to define an operator on cohomology with $\Z/2\Z$ coefficients that we call the second Steenrod Square, since if the framed $1$-flow category comes from a framed flow category, the operator agrees with the second Steenrod Square of the corresponding stable homotopy type.

To define this operator, assume that $(\cC, s, f)$ is a framed $1$-flow category, and $\varphi\in C^k(\cC;\Z/2\Z)$ is a cocycle. Let $c_1,\ldots, c_n\in \Ob(\cC)$ be the objects with $|c_i| = k$ and $\varphi(c_i)=1\in \Z/2\Z$ for all $i\in \{1,\ldots, n\}$.

Now let $b_1,\ldots, b_m\in \Ob(\cC)$ be those objects with $|b_j| = k+1$ and $\M(b_j,c_i)$ non-empty for some $i$. Consider the disjoint union
\[
\M(b_j, \varphi) = \coprod_{i=1}^n \M(b_j, c_i),
\] 
which is a non-empty set with an even number of elements, since $\varphi$ is a cocycle.

\begin{definition}
Let $(\cC,s,f)$ be a framed 1-flow category and $\varphi\in C^k(\cC;\Z/2\Z)$ be a cocycle. Writing $b_1,\dots b_m$ as above we may choose a partition of the elements of $\M(b_j,\varphi)$ into ordered pairs. If we make this choice for each $j=1,\ldots,m$, the overall choice is called a \em combinatorial boundary matching \em $\mathcal{C}$ for $\varphi$.
\end{definition}

\begin{definition}\label{def:graphstructure}
 A \em special graph structure \em $\Gamma=\Gamma(V,E,E',E'',L,s,f)$ consists of a set of vertices $V$, a function $s\colon V \to \Z/2\Z$, a set of edges $E$, a subset $E'\subset E$ together with a function $f\colon E'\to \Z/2\Z$, a subset $E'' \subset E-E'$ of directed edges, and a set of loops $L$. All sets are finite. The following conditions are satisfied:
\begin{enumerate}
 \item \label{item:gs1}Each vertex is contained in at least one edge, and in at most two edges. We denote the set of vertices contained in only one edge by $\partial V$ and call these the boundary points.
 \item \label{item:gs2}If $v\in \partial V$, then the unique edge $e(v)$ with $v\in e(v)$ satisfies $e(v)\in E'$.
 \item \label{item:gs3}If $v \in V-\partial V$, then there is an edge $e_1\in E'$ with $v\in e_1$, and an edge $e_2\in E-E'$ with $v\in e_2$.
 \item \label{item:gs4}If $e\in E'$ and $e=\{v_1,v_2\}$, then $s(v_1)\not=s(v_2)$.
 \item \label{item:gs5}If $e\in E-E'$ and $e=\{v_1,v_2\}$, then $s(v_1)=s(v_2)$ if and only if $e\in E''$.
\end{enumerate}
The edges $E$ determine an equivalence relation on the vertex set $V$ in the obvious way. By condition (1) the edges in an equivalence class either form an interval or a circle. The set of equivalence classes under this relation, taken together with the set $L$ of loops, forms the set of \em components \em of $\Gamma$. 
\end{definition}

Notice that $L$ is a set independent from possible circles made from edges in $E$.

\begin{example}\label{ex:cocycle}
Let $(\cC, s, f)$ be a framed $1$-flow category, $\varphi\in C^k(\cC;\Z/2\Z)$ a cocycle, and $\mathcal{C}$ a combinatorial boundary matching for $\varphi$. Let $c_1,\ldots, c_n, b_1,\ldots, b_m\in \Ob(\cC)$ be as above.

Given $a\in \Ob(\cC)$ with $|a|=k+2$, define a special graph structure $\Gamma_\mathcal{C}(a,\varphi)$ as follows. The vertex set $V$ is the disjoint union
\[
V = \coprod_{i,j} \M(b_j, c_i)\times \M(a,b_j),
\]
and $s\colon V \to \Z/2\Z$ is given by $s(B,A) = s(B)+s(A)$ with the latter $s$ the sign assignment of $\cC$. Each interval component $I\subset \M(a, c_i)$ defines an edge $e_I\in E'$, with $f\colon E'\to \Z/2\Z$ determined by the framing of $\cC$. Each circle component $C\subset \M(a,c_i)$ with $f(C) = 0$ defines a loop in $L$.

The remaining edges $E-E'$ are determined by $\mathcal{C}$. If $(B_1, B_2)\in \mathcal{C}$, there is a unique $j$ with $B_1,B_2\in \M(b_j,\varphi)$. For $A\in \M(a,b_j)$ we now get an edge $e\in E-E'$ between $(B_1, A)$ and $(B_2, A)$. If $s(B_1) = s(B_2)$, then $e\in E''$, and the direction is from $(B_1, A)$ to $(B_2, A)$, that is, the direction is inherited from $\mathcal{C}$.

Notice that $\partial V = \emptyset$.
\end{example}

\begin{definition}
Let $\Gamma=\Gamma(V,E,E',E'',L,s,f)$ be a special graph structure. For a circle component $C$ of $\Gamma$ which contains a vertex, let $F(C)$ be the sum of the framing values $f(e')$ where $e'\in E'$ is in $C$. Also, let $D(C)\in \Z/2\Z$ be the number of oriented edges in $C$ which point in a chosen given direction.
\end{definition}

It is shown in \cite[Lm.3.10]{MR4165986} that $D(C)$ does not depend on the direction.

\begin{definition}\label{def:sqca}
Let $(\cC,s,f)$ be a framed $1$-flow category and $\varphi\in C^k(\cC;\Z/2\Z)$ a cocycle. Then define a cochain $\sq^\varphi\colon C_{k+2}(\cC;\Z/2\Z)\to \Z/2\Z$ by 
\[
\sq^\varphi(a):=|L|+\sum_{C}\left(1+F(C)+D(C)\right)\in\Z/2\Z,
\]
where $L$ is the loop set of $\Gamma_\mathcal{C}(a,\varphi)$ and the sum is taken over all components $C$ of~$\Gamma_\mathcal{C}(a,\varphi)$ containing a vertex. Here $\Gamma_\mathcal{C}(a,\varphi)$ is the special graph structure from Example \ref{ex:cocycle}.
\end{definition}

By \cite[Thm.3.13]{MR4165986} $\sq^\varphi$ is a cocycle whose cohomology class does not depend on the combinatorial matching, and there is a linear map
\[
\Sq^2\colon H^k(\cC;\Z/2\Z) \to H^{k+2}(\cC;\Z/2\Z);\qquad z\mapsto [\sq^\varphi],
\]
where $\varphi$ is any choice of cocycle representing the cohomology class $z$.

\subsection{Short exact sequences of $1$-flow categories}\label{ssc:ses}
We need a bit of functoriality for the $\Sq^2$-operator. It will be enough to look at sub- and quotient $1$-flow categories.

\begin{definition}
Let $\cC$ and $\cC'$ be $1$-flow categories, and assume that $\Ob(\cC')\subset \Ob(\cC)$, with $|\cdot|_{\cC'}$ the restriction of $|\cdot|_\cC$. We call $\cC'$ a {\em full subcategory of $\cC$}, if $\M_{\cC'}(a,b) = \M_{\cC}(a,b)$ for all $a,b\in \Ob(\cC')$. If $\cC, \cC'$ are framed, we also require the sign and framing assignment of $\cC'$ to be restriction.
\end{definition}

The next definition is the $1$-flow category version of \cite[Def.3.31]{MR3230817}.

\begin{definition}
Let $\cC'$ be a full subcategory of the $1$-flow category $\cC$. We say that $\cC'$ is a {\em downward closed subcategory} (respectively, an {\em upward closed subcategory}) of $\cC$, if for all $x,y\in \Ob(\cC)$ with $\M_\cC(x,y)\not=\emptyset$, $x\in \Ob(\cC')$ implies $y\in \Ob(\cC')$ (respectively, $y\in \Ob(\cC')$ implies $x\in \Ob(\cC')$).
\end{definition}

Let $\cC$ be a $1$-flow category, and $\cC'$ a downward closed subcategory. Define $\cC''$ by $\Ob(\cC'') = \Ob(\cC) - \Ob(\cC')$, and if $x,y\in \Ob(\cC'')$, let $\M_{\cC''}(x,y) = \M_{\cC}(x,y)$. All other functions (grading, sign and framing (if $\cC$ is framed)) are given by restriction. It is then easy to see that $\cC''$ is an upward closed subcategory of $\cC$.

Similarly, if $\cC''$ is an upward closed subcategory of $\cC$, we can define a downward closed subcategory $\cC'$ in the same way.

\begin{definition}
Let $\cC$ be a $1$-flow category, and $\cC', \cC''$ full subcategories with $\Ob(\cC')\sqcup \Ob(\cC'') = \Ob(\cC)$. If $\cC'$ is downward closed and $\cC''$ upward closed in $\cC$, we write
\[
0 \longrightarrow \cC' \longrightarrow \cC \longrightarrow \cC'' \longrightarrow 0
\]
and call this a short exact sequence of $1$-flow categories.
\end{definition}

If $\cC$ is framed, so are $\cC'$ and $\cC''$, and we get a short exact sequence of the corresponding chain complexes. In particular, there is a long exact sequence of cohomology groups
\begin{multline*}
\cdots \longrightarrow H^{k-1}(\cC';\Z/2\Z) \stackrel{\delta^\ast}{\longrightarrow} \\ 
H^k(\cC'';\Z/2\Z) \stackrel{p^\ast}{\longrightarrow} H^k(\cC;\Z/2\Z) \stackrel{i^\ast}{\longrightarrow} H^k(\cC'';\Z/2\Z) \longrightarrow \cdots
\end{multline*}

\begin{lemma}\label{lm:easyfunc}
Let
\[
0 \longrightarrow \cC' \longrightarrow \cC \longrightarrow \cC'' \longrightarrow 0
\]
be a short exact sequence of framed $1$-flow categories. Then
\[
p^\ast\circ \Sq^2 = \Sq^2\circ p^\ast\colon H^k(\cC'';\Z/2\Z)\to H^{k+2}(\cC;\Z/2\Z),
\]
and
\[
i^\ast\circ \Sq^2 = \Sq^2\circ i^\ast\colon H^k(\cC;\Z/2\Z)\to H^{k+2}(\cC';\Z/2\Z),
\]
for all $k\in \Z$.
\end{lemma}

\begin{proof}
Let $\varphi\in C^k(\cC'';\Z/2\Z)$ be a cocycle, and $\mathcal{C}$ a combinatorial boundary matching for $\varphi$. Denote $c_1,\ldots, c_n\in \Ob(\cC'')$ the objects with $\varphi(c_i) = 1$. Let $b_1,\ldots, b_m\in \Ob(\cC'')$ be the objects with $\M_{\cC''}(b_j,c_i) \not=\emptyset$. We have $\Ob(\cC'')\subset \Ob(\cC)$, and the $c_1,\ldots, c_n$ are exactly the objects $c$ in $\Ob(\cC)$ with $p^\ast\varphi(c) = 1$. 
Furthermore, since $\cC''$ is upward closed, there are no other $b\in \Ob(\cC)$ with $\M_\cC(b,c_j)\not=\emptyset$. Hence $\mathcal{C}$ also works as a combinatorial boundary matching for $p^\ast\varphi$.

If $a\in \Ob(\cC)$, we either have $a\in \Ob(\cC')$ or $a\in \Ob(\cC'')$. In the first case the special graph structure $\Gamma_\mathcal{C}(a,p^\ast\varphi)$ from Example \ref{ex:cocycle} is empty, as each $\M_\cC(a,c_i)=\emptyset$. If $a\in \Ob(\cC'')$, then $\Gamma_\mathcal{C}(a,\varphi)$ agrees with $\Gamma_\mathcal{C}(a,p^\ast\varphi)$, and therefore $\sq^\varphi(a) = \sq^{p^\ast\varphi}(a)$ for these $a$. Hence $\sq^{p^\ast\varphi} = p^\ast\sq^\varphi$, and $\Sq^2$ commutes with $p^\ast$.

The case for $i^\ast$ is similar, and will be omitted.
\end{proof}

To see that $\Sq^2$ also commutes with the connecting homomorphism, let us first introduce a mapping cone $1$-flow category.

Let $\cC'$ be a downward closed subcategory of the framed $1$-flow category $\cC$ and write $0\longrightarrow\cC'\stackrel{i}{\longrightarrow}\cC$. Let $\CC\cC_i$ be the following $1$-flow category. We set $\Ob(\CC\cC_i) = \Ob(\cC)\sqcup \Ob(\cC')$. Since $a\in \Ob(\cC')\subset \Ob(\cC)$ appears twice in $\Ob(\CC\cC_i)$, let us write $\bar{a}\in \Ob(\CC\cC_i)$ if we mean it to be in the second copy, and $a\in \Ob(\CC\cC_i)$ if we mean it to be in the first copy.

The grading is given by $|a|_{\CC\cC_i} = |a|_\cC$ for $a\in \Ob(\cC)$, and for $a\in \Ob(\cC')$ we set $|\bar{a}|_{\cCC_i} = |a|_{\cC'}+1$. 

For $a,b\in \Ob(\cC)$, let
\[
\M_{\CC\cC_i}(a,b) = \M_\cC(a,b),
\]
and for $a,b\in \Ob(\cC')$ let
\[
\M_{\CC\cC_i}(\bar{a},\bar{b}) = \M_{\cC'}(a,b).
\]
If $a\in \Ob(\cC')$, let
\[
\M_{\CC\cC_i}(\bar{a},a) = \{P_a\},
\]
a one-element set. Also, if $a,b\in \Ob(\cC')$ with $|a| = |b|+1$, let
\[
\M_{\CC\cC_i}(\bar{a},b) = \M_{\cC'}(a,b) \times I.
\]
If $A\in \M_{\cC'}(a,b)$ we write $I_A$ for the interval $\{A\}\times I$, and we stipulate that $\partial I_A = \{(P_b, A), (A, P_a)\}$. Any $\M_{\CC\cC_i}(a,\bar{b})$ are empty. It is easy to see that $\CC\cC_i$ is a $1$-flow category.

If $(\cC,s,f)$ is framed, we can frame $\CC\cC_i$ as follows. For $\M_{\CC\cC_i}(a,b)$ and $\M_{\CC\cC_i}(\bar{a},\bar{b})$ we use the same signs and framing as in $\cC$. For $P_a\in \M_{\CC\cC_i}(\bar{a},a)$ we set
\[
s(P_a) = |a|_{\cC'},
\]
and for $A\in \M_{\cC'}(a,b)$ we set
\[
f(I_A) = s(A) \cdot (|b|_{\cC'} + s(A))\in \Z/2\Z.
\]

\begin{lemma}
The $1$-flow category $\CC\cC_i$ together with $s$ and $f$ is framed.
\end{lemma}

\begin{proof}
It is easy to see that $s$ is indeed a sign assignment. To see that $f$ is a framing, we only need to consider the case of $a,c\in \Ob(\cC')$ with $|a| = |c|+2$. Each circle component $C\subset \M_{\cC'}(a,c)$ leads to two circles $\{P_c\}\times C$ and $C\times \{P_a\}$ in $\partial \M_{\CC\cC_i}(\bar{a},c)$ whose framing contributions cancel each other.

If $I$ is an interval in $\M_{\cC'}(a,c)$ with boundary points $(B_j,A_j)\in \M_{\cC'}(b_j,c)\times \M_{\cC'}(a,b_j)$ for $j=1,2$, the following six intervals form a hexagon in $\partial \M_{\CC\cC_i}(\bar{a},c)$:
\begin{align*}
\{P_c\}\times I & \mbox{ with boundary } (P_c, B_1, A_1) , (P_c, B_2, A_2)\\
I_{B_2}\times \{A_2\} & \mbox{ with boundary } (P_c, B_2, A_2), (B_2, P_{b_2}, A_2) \\
\{B_2\}\times I_{A_2} & \mbox{ with boundary } (B_2, P_{b_2}, A_2), (B_2, A_2, P_a) \\
I \times \{P_a\} & \mbox{ with boundary } (B_2, A_2, P_a), (B_1, A_1, P_a) \\
\{B_1\}\times I_{A_1} & \mbox{ with boundary } (B_1, A_1, P_a), (B_1, P_{b_1}, A_1) \\
I_{B_1}\times \{A_1\} & \mbox { with boundary } (B_1, P_{b_1}, A_1), (P_ c, B_1, A_1)
\end{align*}
The framing values are given by
\begin{align*}
f(I_{B_2}) &= s(B_2)(|c|+s(B_2)) \\
f(I_{A_2}) &= s(A_2)(|b_2|+s(A_2)) \\
f(I_{B_1}) & = s(B_1)(|c|+s(B_1)) \\
f(I_{A_1}) & = s(A_1)(|b_1|+s(A_1)).
\end{align*}
We also get twice $f(I)$, but this will cancel. To satisfy the compatibility condition, we need
\[
s(P_c)+s(B_1)+s(B_2)+f(I_{A_1}) + f(I_{A_2}) + f(I_{B_1}) + f(I_{B_2}) = 0.
\]
Using that $s(P_c) = |c| = |b_j|+1\in \Z/2\Z$, we have
\begin{align*}
\sum_{j=1}^2f(I_{A_j}) + f(I_{B_j})  &= \sum_{j=1}^2 s(A_j)(|c|+1+s(A_j))+s(B_j)(|c|+s(B_j))\\
&= |c| \cdot (s(A_1)+s(A_2)+s(B_1)+s(B_2)) + s(B_1)+s(B_2) \\
& = |c| + s(B_1)+s(B_2),
\end{align*}
since $s$ is a sign assignment and $u(1+u) = 0$ for all $u\in \Z/2\Z$. Therefore each component of $\partial \M_{\CC\cC_i}(\bar{a},c)$ satisfies the compatibility condition, and $f$ is a framing for $\CC\cC_i$.
\end{proof}

We call $\CC\cC_i$ the {\em mapping cone} of the inclusion $i\colon \cC'\to \cC$. It is easy to see that we have a short exact sequence of $1$-flow categories
\[
0\longrightarrow \cC \longrightarrow \CC\cC_i \longrightarrow \mathfrak{S}\cC'\longrightarrow 0
\]
where $\mathfrak{S}\cC'$ is the {\em suspension} of $\cC'$, a $1$-flow category which agrees with $\cC'$ except that for $a\in \Ob(\mathfrak{S}\cC')=\Ob(\cC')$ we have $|a|_{\mathfrak{S}\cC'} =|a|_{\cC'}+1$.

If $\cC''$ is the upward closed subcategory of $\cC$ fitting into a short exact sequence
\[
0 \longrightarrow \cC' \longrightarrow \cC \longrightarrow \cC'' \longrightarrow 0,
\]
it follows from standard homological algebra that for any abelian group $G$
\[
H^k(\CC\cC_i;G) \cong H^k(\cC'';G),
\]
and the connecting homomorphism $\delta^\ast\colon H^k(\cC';G) \to H^{k+1}(\cC'';G)$ agrees with
\[
H^k(\cC';G) \cong H^{k+1}(\mathfrak{S}\cC';G) \stackrel{p^\ast}{\longrightarrow} H^{k+1}(\CC\cC_i;G) \cong H^{k+1}(\cC'';G).
\]
From Lemma \ref{lm:easyfunc} we get naturality of $\Sq^2$ with the connecting homomorphism.

\begin{lemma}\label{lm:funcsqconnect}
Let
\[
0 \longrightarrow \cC' \longrightarrow \cC \longrightarrow \cC'' \longrightarrow 0
\]
be a short exact sequence of framed $1$-flow categories. Then
\[
\delta^\ast\circ \Sq^2 = \Sq^2\circ \delta^\ast\colon H^{k-1}(\cC';\Z/2\Z)\to H^{k+2}(\cC'';\Z/2\Z),
\]
for all $k\in \Z$. \hfill\qed
\end{lemma}

\subsection{Signed covers of the cube $1$-flow category}

We continue our analogous constructions for $1$-flow categories from \cite{MR3230817} by looking at covers. We will only consider covers over the cube $1$-flow category, but we generalize those by allowing more flexibility with signs.

\begin{definition}\label{def:cover}
Denote by $(\cC(n),s,f)$ the cube $1$-flow category with a chosen sign and frame assignment. Also, let $\delta,\varepsilon\in \Z/2\Z$. A framed $1$-flow category $(\cC,s_\cC,f_\cC)$ is called a {\em signed cover of type $(\delta,\varepsilon)$ of $(\cC(n),s,f)$}, if there is a grading preserving function $h\colon \Ob(\cC)\to \Ob(\cC(n))$ such that the following hold:
\begin{enumerate}
\renewcommand{\theenumi}{\alph{enumi}}
\item For all objects $a,b\in \Ob(\cC)$ with $|a|=|b|+1$ there is a covering map $h_0\colon \M_\cC(a,b)\to \M_{\cC(n)}(h(a),h(b))$.
\item For all objects $a,c\in \Ob(\cC)$ with $|a|=|c|+2$ there is a covering map $h_1\colon \M_\cC(a,c) \to \M_{\cC(n)}(h(a),h(c))$ such that
\[
\begin{tikzpicture}
\node at (0,0) {$\M_\cC(a,c)$};
\node at (7,0) {$\M_{\cC(n)}(h(a),h(c))$};
\node at (0,2) {$\bigcup_b \M_\cC(b,c)\times \M_\cC(a,b)$};
\node at (7,2) {$\bigcup_v \M_{\cC(n)}(v,h(c))\times \M_{\cC(n)}(h(a),v)$};
\draw[right hook ->] (0,1.7) -- (0, 0.3);
\draw[right hook ->] (7,1.7) -- (7, 0.3);
\draw[->] (2.1, 2) -- node [above] {$h_0\times h_0$} (4.1, 2);
\draw[->] (1, 0) -- node [above] {$h_1$} (5.4, 0);
\end{tikzpicture}
\]
commutes.
\item For all objects $a,d\in \Ob(\cC)$ with $|a| = |d|+3$ the induced map
\[
h_\ast\colon \partial\M_\cC(a,d) \to \partial \M_{\cC(n)}(h(a),h(d))
\]
is a trivial covering map.
\item For each interval component $I\subset \M_\cC(a,c)$ the framing value $f_\cC(I)$ is obtained from $f(h_1(I))$ via (\ref{eq:framechange}).
\end{enumerate} 
\end{definition}

The sign and framing are only used in part (d), and $f_\cC$ is uniquely determined by (d). In fact, given a sign assignment $s_\cC$ we can use (\ref{eq:framechange}) to define a pre-framing, and this pre-framing is automatically a framing:

\begin{lemma}\label{lm:framecover}
Let $(\cC(n),s,f)$ be a framed cube $1$-flow category, and $(\cC,s_\cC)$ a $1$-flow category with sign assignment $s_\cC$ and $h\colon \Ob(\cC)\to \Ob(\cC(n))$ a grading preserving function such that (a), (b), (c) in Definition \ref{def:cover} are satisfied. The pre-framing $f_\cC$ obtained via (\ref{eq:framechange}) to satisfy (d) turns $(\cC,s_\cC,f_\cC)$ into a framed $1$-flow category which is a signed cover of $(\cC(n),s,f)$.
\end{lemma}

\begin{proof}
We need to check the compatibility condition. From the definition it follows that any component $C$ in $\partial_\cC(a,d)$, where $|a|=|d|+3$, is a hexagon. Furthermore, there exist $b_i,c_i\in \Ob(\cC)$, $C_i\in \M_\cC(c_i,d)$, $A_i\in \M_\cC(a,b_i)$, and $B_{ji}\in \M_\cC(b_i,c_j)$ for $i,j\in \{1,2,3\}$ with $j\not=i$, such that we have a sub-cube $\mathcal{C}$
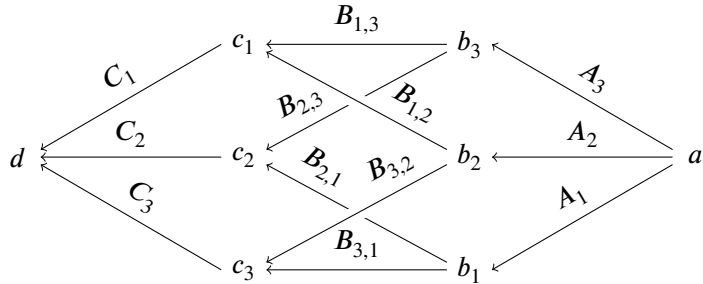
\[
\begin{tikzpicture}
\node at (0,0) {$d$};
\node at (3,1.5) {$c_1$};
\node at (3,0) {$c_2$};
\node at (3,-1.5) {$c_3$};
\node at (6,1.5) {$b_3$};
\node at (6,0) {$b_2$};
\node at (6,-1.5) {$b_1$};
\node at (9,0) {$a$};
\draw[<-] (0.3, 0.1) -- node [above, sloped] {$C_1$} (2.7,1.5);
\draw[<-] (0.3, 0) -- node [above, sloped] {$C_2$} (2.7, 0);
\draw[<-] (0.3, -0.1) -- node [above, sloped] {$C_3$} (2.7, -1.5);
\draw[<-] (3.3, 0.1) -- node [above, sloped, near start] {$B_{2,3}$} (5.7, 1.4);
\draw[<-] (3.3, -0.1) -- node [above, sloped, near start] {$B_{2,1}$} (5.7, -1.4);
\draw[-, white, line width = 6pt] (3.3, 1.4) -- (5.7, 0.1);
\draw[-, white, line width = 6pt] (3.3, -1.4) -- (5.7, -0.1);
\draw[<-] (3.3, 1.5) -- node [above] {$B_{1,3}$} (5.7, 1.5);
\draw[<-] (3.3, 1.4) -- node [above, sloped, near end] {$B_{1,2}$} (5.7, 0.1);
\draw[<-] (3.3, -1.4) -- node [above, sloped, near end] {$B_{3,2}$} (5.7, -0.1);
\draw[<-] (3.3, -1.5) -- node [above] {$B_{3,1}$} (5.7, -1.5);
\draw[<-] (6.3, 1.5) -- node [above, sloped] {$A_3$} (8.7, 0.1);
\draw[<-] (6.3, 0) -- node [above] {$A_2$} (8.7, 0);
\draw[<-] (6.3, -1.5) -- node [above, sloped] {$A_1$} (8.7, -0.1);
\end{tikzpicture}
\]
covering a $3$-dimensional sub-cube $C_{u,x}$ of $\cC(n)$. The sign assignment $s_\cC$ restricted to $\mathcal{C}$ induces a sign assignment $s'$ on $C_{u,x}$. The framing values for the intervals in $\mathcal{C}$ are obtained via (\ref{eq:framechange}) and the change of signs on $C_{u,x}$ from $s$ to $s'$. By Lemma \ref{lm:frame4signchange} the component $C$ contributes $0$ to the compatibility condition.
\end{proof}

\section{The signed Burnside category}

In order to define a stable homotopy type for odd Khovanov homology, Sarkar--Scaduto--Stoffregen \cite{MR4078823} introduced the signed Burnside category. We are going to recall their definition.

\subsection{The definition of the signed Burnside category}
\label{ss:burn}
Let $X$ and $Y$ be finite sets. A {\em signed correspondence} between $X$ and $Y$ is a tuple $(A, s_A, t_A, \sigma_A)$ with $A$ a finite set and functions $s_A\colon A\to X$, $t_A\colon A\to Y$ and $\sigma_A\colon A\to \Z/2\Z$. The function $\sigma_A$ is called the {\em sign} of the signed correspondence. Notice that in \cite{MR4078823} the codomain for $\sigma_A$ is $\{\pm 1\}$.

To simplify notation we will often suppress the functions and denote a signed correspondence by the set $A$ only. If $A$ is a signed correspondence from $X$ to $Y$ and $B$ a signed correspondence from $Y$ to $Z$, their composition $B\circ A$ is the signed correspondence $(C,s,t,\sigma)$ given by
\[
C = B\times_Y A = \{(b,a)\in B\times A\mid t_A(a) = s_B(b)\},
\]
with $s(b,a) = s_A(a)$, $t(b,a) = t_B(b)$ and $\sigma(b,a) = \sigma_B(b)+\sigma_A(a)$.

The identity correspondence is given by $(X,\id_X, \id_X, \sigma_0)$ where $\sigma_0\colon X\to \Z/2\Z$ is the constant function $\sigma_0(x) = 0$.

Given two signed correspondences $A, B$ between $X$ and $Y$, a {\em map of signed correspondences} $f\colon A\to B$ is a bijection of sets $f\colon A\to B$ with $s_A = s_B\circ f$, $t_A=t_B\circ f$ and $\sigma_A=\sigma_B\circ f$.

\begin{definition}
The {\em signed Burnside category} $\mathcal{B}_\sigma$ is the weak $2$-category whose objects are finite sets, morphisms are given by signed correspondences, and $2$-morphisms are given by maps of signed correspondences.
\end{definition}

Notice that composition of morphisms is not strictly associative, and composition by the identity correspondence does not fix a signed correspondence. For this reason we resort to weak $2$-categories, see \cite{MR0220789} for more information on them.

To resolve these issues we use the natural $2$-morphisms
\[
\lambda\colon Y\times_Y A \to A \hspace{1cm} \rho\colon A\times_X X \to A,
\]
the {\em left and right unitors} given by $\lambda(y,a) = a$ and $\rho(a,y) = a$. Furthermore, if we have signed correspondences $A$ between sets $W$ and $X$, $B$ between $X$ and $Y$, and $C$ between $Y$ and $Z$, there is the {\em associator}
\[
\alpha\colon (C\times_Y B)\times_X A \to C\times_Y(B\times_X A)
\]
given by $\alpha((c,b),a) = (c,(b,a))$. 

The Burnside category $\mathcal{B}$ as defined in \cite{MR4153651} is the subcategory of $\mathcal{B}_\sigma$ where the signed correspondences have sign functions constant $0$. Alternatively, one can repeat the definition of correspondences without sign function. There is the forgetful functor from $\mathcal{B}_\sigma$ to $\mathcal{B}$ which turns sign functions to $0$. 

\subsection{Functors from the cube category to the signed Burnside category}
\label{sb:functorBurn}
We are only going to consider functors $F\colon \cn \to \Bu$ which are strictly unitary, lax $2$-functors. Such functors are determined by the following data.
\begin{enumerate}
\item For each object $v\in \cn$ an object $F(v) = F_v\in \Ob(\mathcal{B}_\sigma)$.
\item For any $u\geq v$ a signed correspondence $F(\phi_{u,v}) = F_{u,v}$ from $F_u$ to $F_v$ such that $F_{u,u}$ is the identity correspondence.
\item For any $u\geq v\geq w$ a map of signed correspondences $F_{u,v,w}$ from $F_{v,w}\circ F_{u,v}$ to $F_{u,w}$ that agrees with $\lambda$ (respectively, $\rho$) when $v=w$ (respectively, $u=v$), and that satisfies for any $u\geq v\geq w\geq x$ the following commutes.
\[
\begin{tikzpicture}
\node at (0,4) {$(F_{w,x}\times_{F_w} F_{v,w})\times_{F_v}F_{u,v}$};
\node at (6,4) {$F_{w,x}\times_{F_w} (F_{v,w}\times_{F_v}F_{u,v})$};
\node at (0,2) {$F_{v,x}\times_{F_v} F_{u,v}$};
\node at (6,2) {$F_{w,x}\times_{F_w} F_{u,w}$};
\node at (3,0.5) {$F_{u,x}$};
\draw[->] (2.1,4) -- node [above] {$\alpha$} (3.9,4);
\draw[->] (0, 3.7) -- node [right] {$F_{v,w,x}\times \id$} (0, 2.3);
\draw[->] (6, 3.7) -- node [left] {$\id \times F_{u,v,w}$} (6, 2.3);
\draw[->] (0.2, 1.7) -- node [above, sloped] {$F_{u,v,x}$} (2.6, 0.6);
\draw[<-] (3.4, 0.6) -- node [above, sloped] {$F_{u,w,x}$} (5.8, 1.7); 
\end{tikzpicture}
\]
\end{enumerate}
See \cite[\S 3.5]{MR4078823} and \cite[\S 4]{MR4153651} for how this fits in with the general theory of functors between weak $2$-categories.

In \cite[\S 4.3]{MR4153651} it is shown how to construct a framed flow category from a functor $F\colon \cn\to \mathcal{B}$ to the (unsigned) Burnside category. We now adapt this construction to get a framed $1$-flow category $\cC_F$ from a functor $F\colon \cn \to \mathcal{B}_\sigma$. This $1$-flow category is going to be a signed cover of $\cC(n)$ with a fixed sign and frame assignment $s$ and $f$.

The object set $\Ob(\cC_F)$ is the disjoint union of all $F_v$ with $v\in \Ob(\cn)$ and the grading $|\cdot |\colon \Ob(\cC_F)\to\Z$ restricted to $F_v$ is $|v|$. Sending $F_v$ to $v$ gives the grading preserving function $h$ on objects.

If $a,b\in \Ob(\cC_F)$ with $a\in F_u$ and $b\in F_v$ with $u\geq_1 v$, we set
\[
\M_{\cC_F}(a,b) = s_{F_{u,v}}^{-1}(\{a\}) \cap t_{F_{u,v}}^{-1}(\{b\}) \subset F_{u,v}.
\]
All other $0$-dimensional moduli spaces are empty. Clearly condition (a) in Definition \ref{def:cover} is satisfied.

For $P\in \M_{\cC_F}(a,b)$ the sign assignment $s_F$ for $\cC_F$ is defined by
\[
s_F(P) = s(C_{u,v})+\sigma_{F_{u,v}}(P)\in \Z/2\Z.
\]
If $a,c\in \Ob(\cC_F)$ with $a\in F_u$ and $c\in F_w$ with $u\geq_2 w$, we set
\[
\M_{\cC_F}(a,c) = s_{F_{u,w}}^{-1}(\{a\}) \cap t_{F_{u,w}}^{-1}(\{c\}) \times \M_{\cC(n)}(u,w)\subset F_{u,w} \times \M_{\cC(n)}(u,w).
\]
Notice that $\M_{\cC(n)}(u,w)$ is a compact interval. All other $1$-dimensional moduli spaces are empty. From the functoriality of $F$ we get that (b) in Definition \ref{def:cover} is satisfied.

If $u\geq_2 w$, then there exist exactly two objects $v_1,v_2$ with $u\geq_1 v_1,v_2 \geq_1 w$. Furthermore, for $i=1,2$ we have the bijections $F_{u,v_i,w}\colon F_{v_i,w}\times_{F_{v_i}} F_{u,v_i} \to F_{u,w}$ which determines a bijection 
\[
\coprod_{b\in F_{{v_i}}}\M_{\cC_F}(b,c)\times M_{\cC_F}(a,b) \to s_{F_{u,w}}^{-1}(\{a\}) \cap t_{F_{u,w}}^{-1}(\{c\})
\]
for all $a\in F_u$, $c\in F_w$. Combining these for $i=1$ and $2$ determines $\partial \M_{\cC_F}(a,c)$. Notice that points $(B,A)\in s^{-1}(\{a\})\cap t^{-1}(\{c\})\subset F_{v_i,w}\times_{F_{v_i}} F_{u,v_i}$ with $t_{F_{v_i,w}}(B) = b = s_{F_{u,v_i}}(A)$ are also in $\M_{\cC_F}(b,c)\times M_{\cC_F}(a,b)$.

Each $X\in s_{F_{u,w}}^{-1}(\{a\}) \cap t_{F_{u,w}}^{-1}(\{c\})$ gives an interval component in $\M_{\cC_F}(a,c)$, and if we write $X_i = F^{-1}_{u,v_i,w}(X)$ for $i=1,2$, these two points determine two points $(B_i,A_i)\in \M_{\cC_F}(b_i,c)\times \M_{\cC_F}(a,b_i)$ with $b_i\in F_{v_i}$ that represent the boundary of the interval.

Functoriality ensures that
\[
\sigma_{F_{v_1,w}}(B_1) + \sigma_{F_{u,v_1}}(A_1) = \sigma_{F_{v_2,w}}(B_2) + \sigma_{F_{u,v_2}}(A_2) = \sigma_{F_{u,w}}(X),
\]
so that $s_F$ is indeed a sign assignment.

Furthermore, the $\sigma$ values of the four points $B_1,A_1,B_2,A_2$ determine how $s$ differs from $s_F$ on this two-dimensional cube. Note that $s_F$ gives a sign assignment on the 2-dimensional sub-cube $C_{u,w}$. We can extend $s_F$ to a sign assignment of $\cC(n)$, for example, by treating the two coordinates in $C_{u,w}$ as the first two coordinates of the $n$-dimensional cube, and then using the usual product formula as in the standard assignment $s^\ast$.

For a fixed choice of $\delta, \varepsilon\in \Z/2\Z$ we frame the interval $I_X$ corresponding to $X\in s_{F_{u,w}}^{-1}(\{a\}) \cap t_{F_{u,w}}^{-1}(\{c\})$ according to (\ref{eq:framechange}).

By Lemma \ref{lm:framecover} it remains to show that (c) in Definition \ref{def:cover} is satisfied. Let $x,u\in \Ob(\cn)$ with $u\geq_3 x$. There exist $v_1,v_2,v_3,w_1,w_2,w_3\in \Ob(C_{u,x})$ which together with $x,u$ span this subcube.
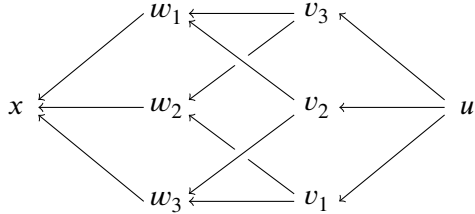
\[
\begin{tikzpicture}
\node at (0,0) {$x$};
\node at (2,1.25) {$w_1$};
\node at (2,0) {$w_2$};
\node at (2,-1.25) {$w_3$};
\node at (4,1.25) {$v_3$};
\node at (4,0) {$v_2$};
\node at (4,-1.25) {$v_1$};
\node at (6,0) {$u$};
\draw[<-] (0.3, 0.1) --  (1.7,1.25);
\draw[<-] (0.3, 0) --  (1.7, 0);
\draw[<-] (0.3, -0.1) -- (1.7, -1.25);
\draw[<-] (2.3, 0.1) --  (3.7, 1.15);
\draw[<-] (2.3, -0.1) --  (3.7, -1.15);
\draw[-, white, line width = 6pt] (2.3, 1.15) -- (3.7, 0.1);
\draw[-, white, line width = 6pt] (2.3, -1.15) -- (3.7, -0.1);
\draw[<-] (2.3, 1.25) --  (3.7, 1.25);
\draw[<-] (2.3, 1.15) -- (3.7, 0.1);
\draw[<-] (2.3, -1.15) -- (3.7, -0.1);
\draw[<-] (2.3, -1.25) --  (3.7, -1.25);
\draw[<-] (4.3, 1.25) --  (5.7, 0.1);
\draw[<-] (4.3, 0) --  (5.7, 0);
\draw[<-] (4.3, -1.25) --  (5.7, -0.1);
\end{tikzpicture}
\]
Let $a\in F_u$ and $d\in F_x$. We need to describe $\partial\M_{\cC_F}(a,d)$. Write
\[
\partial(a,d) = s^{-1}_{F_{u,x}}(\{a\})\cap t^{-1}_{F_{u,x}}(\{d\}) \subset F_{u,x}.
\]
By functoriality, there is a bijection between $\partial(a,d)$ and
\[
s^{-1}(\{a\}) \cap t^{-1}(\{d\}) \subset F_{w_i,x}\times_{F_{w_i}} F_{u,w_i}
\]
for each $i=1,2,3$, and also with
\[
s^{-1}(\{a\}) \cap t^{-1}(\{d\}) \subset F_{v_j,x}\times_{F_{v_j}} F_{u,v_j}
\]
for each $j=1,2,3$. 

In particular, given $H\in \partial(a,d)$ there exists $b_j\in F_{v_j}$ and a $H_{b_j}$ corresponding to an interval component in $\M_{\cC_F}(b_j,d)\times \M_{\cC_F}(a,b_j)$ for each $j\in \{1,2,3\}$, and there exists $c_i\in F_{w_i}$ and $H_{c_i}$ corresponding to an interval component in $\M_{\cC_F}(c_i,d)\times \M_{\cC_F}(a,c_i)$.

If $i\not=j$, we also get a bijection with
\[
s^{-1}(\{a\}) \cap t^{-1}(\{d\}) \subset F_{w_i,x}\times_{F_{w_i}} F_{v_j,w_i}\times_{F_{v_j}} F_{u,v_j},
\]
and by the commuting pentagon in the definition of the functor $F$ the $H\in \partial(a,d)$ corresponds to an element
\[
H_{b_j,c_i}\in \M_{\cC_F}(c_i,d)\times \M_{\cC_F}(b_j,c_i)\times \M_{\cC_F}(a,b_j),
\]
and these elements form the boundaries of the intervals corresponding to the $H_{c_i}$ and $H_{b_j}$. If we write $H_{b_j,c_i} = (C_i,B_{i,j}, A_j)$ for $i\not=j$, we now get a cube 
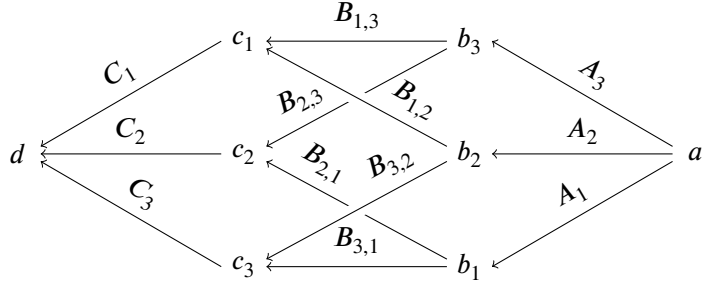
\[
\begin{tikzpicture}
\node at (0,0) {$d$};
\node at (3,1.5) {$c_1$};
\node at (3,0) {$c_2$};
\node at (3,-1.5) {$c_3$};
\node at (6,1.5) {$b_3$};
\node at (6,0) {$b_2$};
\node at (6,-1.5) {$b_1$};
\node at (9,0) {$a$};
\draw[<-] (0.3, 0.1) -- node [above, sloped] {$C_1$} (2.7,1.5);
\draw[<-] (0.3, 0) -- node [above, sloped] {$C_2$} (2.7, 0);
\draw[<-] (0.3, -0.1) -- node [above, sloped] {$C_3$} (2.7, -1.5);
\draw[<-] (3.3, 0.1) -- node [above, sloped, near start] {$B_{2,3}$} (5.7, 1.4);
\draw[<-] (3.3, -0.1) -- node [above, sloped, near start] {$B_{2,1}$} (5.7, -1.4);
\draw[-, white, line width = 6pt] (3.3, 1.4) -- (5.7, 0.1);
\draw[-, white, line width = 6pt] (3.3, -1.4) -- (5.7, -0.1);
\draw[<-] (3.3, 1.5) -- node [above] {$B_{1,3}$} (5.7, 1.5);
\draw[<-] (3.3, 1.4) -- node [above, sloped, near end] {$B_{1,2}$} (5.7, 0.1);
\draw[<-] (3.3, -1.4) -- node [above, sloped, near end] {$B_{3,2}$} (5.7, -0.1);
\draw[<-] (3.3, -1.5) -- node [above] {$B_{3,1}$} (5.7, -1.5);
\draw[<-] (6.3, 1.5) -- node [above, sloped] {$A_3$} (8.7, 0.1);
\draw[<-] (6.3, 0) -- node [above] {$A_2$} (8.7, 0);
\draw[<-] (6.3, -1.5) -- node [above, sloped] {$A_1$} (8.7, -0.1);
\end{tikzpicture}
\]
and this cube gives rise to a hexagon in $\partial\M_{\cC_F}(a,d)$ which trivially covers the corresponding hexagon of the sub-cube $C_{u,x}$.  Varying the elements of $\partial(a,d)$ gives rise to the trivial covering map $\partial\M_{\cC_F}(a,d) \to \partial\M_{\cC(n)}(u,x)$ needed for (c) of Definition \ref{def:cover}.

By Lemma \ref{lm:framecover} we now have a framed $1$-flow category $(\cC_F,s_F,f_F)$.

From this framed $1$-flow category we now get the second Steenrod square. To indicate the various choices we made in the definition of $\cC_F$, let us denote it by $\Sq^2_{s,f,\delta,\varepsilon}$. Luckily, we can remove most of these dependencies. Indeed, this can be done more generally for signed covers. We begin with a construction of a mapping cone.

\begin{definition}
Let $\cC$ be a $1$-flow category. The {\em mapping cone} $\cCC$ is the $1$-flow category with $\Ob(\cCC) = \Ob(\cC)\times\{0,1\}$. We write $a_i=(a,i)$ for $a\in \Ob(\cC)$ and $i\in \{0,1\}$, and set
\[
|a_i|_\cCC = |a|_{\cC}+i.
\]
If $a,b\in \Ob(\cC)$ let
\[
\M_\cCC(a_i,b_i) = \M_{\cC}(a,b)
\]
for $i=0,1$. We also set $\M_\cCC(a_1,a_0) = \{P_a\}$. If $a,b\in \Ob(\cC)$ with $|a|_{\cC}=|b|_{\cC}+1$ we set
\[
\M_\cCC(a_1,b_0) = \M_{\cC}(a,b)\times I.
\]
For $A\in \M_{\cC}(a,b)$ we write $I_A = \{A\}\times I$ for the corresponding interval component.

Moduli spaces of the form $\M_\cCC(a_0,b_1)$ are empty.
\end{definition}

It is easy to see that $\cCC$ is a $1$-flow category, and fits into a short exact sequence of $1$-flow categories
\[
0 \longrightarrow \cC \longrightarrow \cCC \longrightarrow \mathfrak{S}\cC \longrightarrow 0
\]
Indeed, we can think the mapping cone as $\cCC_{\id}$ in the sense of Section \ref{ssc:ses}. Note that we have not used any sign or frame assignments yet. Nevertheless, there is a long exact sequence of cohomology groups with $\Z/2\Z$ coefficients, and the connecting homomorphism
\[
\delta^\ast\colon H^k(\cC_F;\Z/2\Z)\to H^{k+1}(\mathfrak{S}\cC_F';\Z/2\Z)\cong H^k(\cC_F;\Z/2\Z)
\]
is the identity.

Below $\cC$ will have two different framings which can be used to frame the subcategories $\cC$ and $\mathfrak{S}\cC$ of $\cCC$. We then still need to choose a sign for $P_a\in \M_\cCC(a_1,a_0)$ and framings for intervals $I_A\subset \M_\cCC(a_1,b_0)$.

\begin{remark}\label{rm:hexagon}
Assume that a sign assignment $s_\cCC$ and a pre-framing $f_\cCC$ have been chosen on $\cCC$ such that the restriction to $\cC$ and $\mathfrak{S}\cC$ are framings. Checking that $f_\cCC$ is a framing is then reduced to considering $\partial\M_\cCC(a_1,c_0)$, where $a,c\in \Ob(\cC)$ with $|a|_\cC = |c|_\cC+2$. Any interval component $I\subset \M_\cC(a,c)$ gives rise to a hexagon $H_I$ as follows. There exist $b,b'\in \Ob(\cC_F)$ with $\partial I$ the two points $(B,A)\in \M_{\cC_F}(b,c)\times \M_{\cC_F}(a,b)$ and $(B',A')\in \M_{\cC_F}(b',c)\times \M_{\cC_F}(a,b')$. The hexagon is then made up of the six intervals
\begin{align*}
I_0\times \{P_a\}, &\hspace{0.5cm} \{B_0\}\times I_A \\
I_B\times \{A_1\}, &\hspace{0.5cm} \{B'_0\}\times I_{A'}\\
I_{B'}\times \{A'_1\}, &\hspace{0.5cm} \{P_c\}\times I_1,
\end{align*}
and we need to check the sum
\begin{multline*}
t_I = s_\cCC(B_0) + s_\cCC(B'_0) + s_\cCC(P_c) + \\
f_\cCC(I_0) + f_\cCC(I_1) + f_\cCC(I_B) + f_\cCC(I_{B'}) + f_\cCC(I_A) + f_\cCC(I_{A'}).
\end{multline*}
Provided that all $1$-dimensional moduli spaces $\M_\cC(a,c)$ only contain interval components, having all $t_I=0$ is sufficient for the compatibility condition to be satisfied.
\end{remark}

\begin{lemma}\label{lm:nodeltaissue}
The second Steenrod square
\[
\Sq^2_{s,f,\delta,\varepsilon}\colon H^\ast(\cC_F;\Z/2\Z)\to H^{\ast+2}(\cC_F;\Z/2\Z)
\]
does not depend on $\delta\in\Z/2\Z$.
\end{lemma}

\begin{proof}
Let us assume that $(\cC_F,s_F,f_F)$ is a type $(0,\varepsilon)$ signed cover of $(\cC(n),s,f)$, and denote by $(\cC_F',s_F,f'_F)$ the corresponding signed cover of type $(1,\varepsilon)$.

We want to frame the mapping cone $\mathfrak{C}\cC$ so that it fits into a short exact sequence
\begin{equation}\label{eq:conesequence}
0\longrightarrow \cC_F \longrightarrow \mathfrak{C}\cC\longrightarrow \mathfrak{S}\cC_F'\longrightarrow 0
\end{equation}
of framed $1$-flow categories. For this we identify $\cC_F$ with the downward closed subcategory generated by $\Ob(\cC_F)\times \{0\}$, and $\mathfrak{S}\cC_F'$ with the upward closed subcategory generated by $\Ob(\cC_F)\times \{1\}$. The moduli spaces in these subcategories are signed and framed accordingly.

Since $\cC_F$ and $\cC_F'$ use the same sign assignment, we get a sign assignment on $\cCC$ by setting
\[
s_\cCC(P_a) = |a|_{\cC_F}
\]
for $P_a\in \M_\cCC(a_1,a_0)$. Finally, for $A\in \M_{\cC_F}(a,b)$ we set
\[
f_\cCC(I_A) = s_F(A) + s(h_0(A))\cdot |b| \in \Z/2\Z.
\]
Recall that $h$ is the covering map from $\cC_F$ to $\cC(n)$.

%
%
%
%
%
%
%
It remains to show that the compatibility condition is satisfied. Consider a hexagon $H_I$ as in Remark \ref{rm:hexagon}, where $I\subset \M_{\cC_F}(a,c)$ is an interval component. Then
\begin{multline*}
t_I = s_\cCC(B_0) + s_\cCC(B'_0) + s_\cCC(P_c) + \\
f_\cCC(I_0) + f_\cCC(I_1) + f_\cCC(I_B) + f_\cCC(I_{B'}) + f_\cCC(I_A) + f_\cCC(I_{A'}).
\end{multline*}
Notice that $f_\cCC(I_0)+f_\cCC(I_1)$ is $1$, if (\ref{eq:framechange}) involves a $\delta$, and $0$ if it does not. Checking the various cases, we see that
\[
f_\cCC(I_0)+f_\cCC(I_1) = s_F(A)+s(h_0(A))+s_F(A')+s(h_0(A')).
\]
From this we get
\begin{multline*}
t = s_F(B)+s_F(B')+|c|+s_F(A)+s(h_0(A))+s_F(A')+s(h_0(A')) \\
+s_F(B)+s(h_0(B))|c| + s_F(B')+s(h_0(B'))|c|+\\s_F(A)+s(h_0(A))(|c|+1)+s_F(A')+s(h_0(A'))(|c|+1)
\end{multline*}
All of the $s_F(X)$, $X\in \{B,B',A,A'\}$, appear twice so we get
\begin{multline*}
t = |c| (1+s(h_0(B))+s(h_0(B'))+s(h_0(A))+s(h_0(A'))) + \\
s(h_0(A)) + s(h_0(A')) + s(h_0(A)) + s(h_0(A')).
\end{multline*}
Since $s$ is a sign assignment, $t=0$. This shows that the compatibility condition is satisfied, and by Lemma \ref{lm:funcsqconnect} the result follows.
\end{proof}

We will from now on choose $\delta = 0$ in the construction of the Steenrod square.

\begin{lemma}
The second Steenrod square
\[
\Sq^2_{s,f,\varepsilon}\colon H^\ast(\cC_F;\Z/2\Z)\to H^{\ast+2}(\cC_F;\Z/2\Z)
\]
does not depend on the frame assignment $f$ for $\cC(n)$.
\end{lemma}

\begin{proof}
The proof is similar to the proof of Lemma \ref{lm:nodeltaissue}. If we denote by $(\cC'_F,s_F, f'_F)$ the framed $1$-flow category obtained by using a different framing $f'$ for $(\cC(n),s)$, we want to frame $\cCC$ so it fits into (\ref{eq:conesequence}). By Lemma \ref{lm:changefrm} we may assume that $f+f' = \delta(C^\ast_{u,v})$ for a $1$-dimensional sub-cube $C_{u,v}$ of $\cC(n)$.

The sign and frame assignment is the same as in Lemma \ref{lm:nodeltaissue}, except that we frame  intervals $I_A$ in $\M_\cCC(a_1,b_0)$, where $A\in \M_{\cC_F}(a,b)$ by
\[
f_\cCC(I_A) = \left\{
\begin{array}{cl}
s_F(A)\cdot |a| + 1 & \mbox{for }h(a) = u, h(b) = v\\
s_F(A)\cdot |a| & \mbox{otherwise}
\end{array}
\right.
\]
To check the compatibility condition consider an interval $I\subset \M_{\cC_F}(a,c)$ with boundary points  $(B,A)\in \M_{\cC_F}(b,c)\times \M_{\cC_F}(a,b)$ and $(B',A')\in \M_{\cC_F}(b',c)\times \M_{\cC_F}(a,b')$ for some $b,b'\in \Ob(\cC_F)$.

At most one of the four points $B,B',A,A'$ is over the cube $C_{u,v}$, and we have $f_F(I) + f'_F(I) = 1$ if and only if this is the case for one of these points. In this case this point $X$ contributes an extra $1$ in $f_\cCC(I_X)$. A calculation similar to the one in the proof of Lemma \ref{lm:nodeltaissue} now shows that $t_I=0$, and therefore the compatibility condition is satisfied. The result follows as before.
\end{proof}

Before we show that the second Steenrod square does not depend on the choice of sign assignment for $\cC(n)$, let us consider a more general situation.

Let $(\cC, s_\cC, f_\cC)$ be a type $(0,\varepsilon)$-signed cover of $(\cC(n),s,f)$ and let $x\in \Ob(\cC)$. Define a new sign assignment $s_x$ as follows. Let $A\in \M_\cC(a,b)$ with $|a|=|b|+1$. If  $a\not=x\not=b$, then $s_x(A) = s_\cC(A)$. If $x=a$ or $x=b$, let $s_x(A) = s_\cC(A)+1$. It is easy to see that this is indeed a sign assignment on $\cC$.

With this new sign assignment we can define a new frame assignment $f_x$ using (\ref{eq:framechange}) with $\delta=0$. By Lemma \ref{lm:framecover} $(\cC,s_x,f_x)$ is a type $(0,\varepsilon)$-signed cover of $(\cC(n),s,f)$. We call $(\cC,s_x,f_x)$ {\em obtained from $(\cC,s_\cC,f_\cC)$ by a sign change at $x$}.

\begin{lemma}\label{lm:changeofsign}
The second Steenrod square
\[
\Sq^2\colon H^\ast(\cC;\Z/2\Z)\to H^{\ast+2}(\cC;\Z/2\Z)
\]
is the same whether we use  $(\cC, s_\cC, f_\cC)$ or  $(\cC, s_x, f_x)$.
\end{lemma}

\begin{proof}
Again we want a framed $1$-flow category $(\cCC,s_\cCC,f_\cCC)$ with $(\cC,s_\cC,f_\cC)$ as a downward closed subcategory, and $(\cC,s_x,f_x)$ as an upward closed subcategory.  This determines the sign assignment and frame assignment apart from moduli spaces $\M_\cCC(a_1,b_0)$. The remaining signs are given by setting
\[
s_\cCC(P_a) = \left\{
\begin{array}{cc}
|a|_\cC & a\not=x \\
|x|_\cC+1 & a = x
\end{array}
\right.
\]
If $A\in \M_\cC(a,b)$, the framing of $I_A\subset \M_\cCC(a_1,b_0)$ is given by
\[
f_\cCC(I_A) = \left\{
\begin{array}{cc}
s_\cC(A)|a|_\cC & a,b\not= x \\
s_\cC(A)|b|_\cC+\bar{s}(A)\varepsilon & a=x \\
s_\cC(A)|a|_\cC+1+\bar{s}(A)(1+\varepsilon) & b = x,
\end{array}
\right.
\]
where
\[
\bar{s}(A) = s_\cC(A) + s(h_0(A)).
\]
We need to check the compatibility condition, so let $I\subset\M_\cC(a,c)$ be an interval component as in Remark \ref{rm:hexagon}. That is, there exist $b,b'\in \Ob(\cC)$ with $\partial I$ the two points $(B,A)\in \M_{\cC}(b,c)\times \M_{\cC}(a,b)$ and $(B',A')\in \M_{\cC}(b',c)\times \M_{\cC}(a,b')$. Then
\begin{multline*}
t_I = s_x(B)+s_x(B')+s_\cCC(P_c) + f_\cC(I)+f_x(I) \\
+f_\cCC(I_B) +f_\cCC(I_{B'}) +f_\cCC(I_A) +f_\cCC(I_{A'}). 
\end{multline*}
If none of $a,b,b',c$ equals $x$ it is easy to see that $t_I=0$. We now need to consider the various cases where one of those equals $x$.

If $a=x$, we claim that
\begin{equation}\label{eq:a=x}
f_\cC(I) + f_x(I) = s_\cC(A)+s_\cC(A')+(\bar{s}(A)+\bar{s}(A'))\varepsilon
\end{equation}
To see this we need to consider the various cases in (\ref{eq:framechange}) which determine $f_\cC(I)$ and $f_x(I)$. 

If $f_\cC(I)$ is determined by case 1 (resp. case 3) in (\ref{eq:framechange}), then $f_x(I)$ is determined by case 3 (resp. case 1) in (\ref{eq:framechange}). Also, if $f_\cC(I)$ is determined by case 2 (resp. case 8) in (\ref{eq:framechange}), then $f_x(I)$ is determined by case 8 (resp. case 2) in (\ref{eq:framechange}). In each of these cases we have
\[
f_\cC(I)+f_x(I) = s(h_0(A))+s(h_0(A')).
\]
But we also have $\bar{s}(A)+\bar{s}(A')=0$ in these cases, so that
\[
s(h_0(A))+s(h_0(A')) = s_\cC(A)+s_\cC(A') = s_\cC(A)+s_\cC(A')+(\bar{s}(A)+\bar{s}(A'))\varepsilon.
\]
If $f_\cC(I)$ is determined by case 4 (resp. case 6) in (\ref{eq:framechange}), then $f_x(I)$ is determined by case 6 (resp. case 4) in (\ref{eq:framechange}). Also, if $f_\cC(I)$ is determined by case 5 (resp. case 7) in (\ref{eq:framechange}), then $f_x(I)$ is determined by case 7 (resp. case 5) in (\ref{eq:framechange}). In these cases we have $\bar{s}(A)+\bar{s}(A')=1$, and therefore
\begin{align*}
f_\cC(I)+f_x(I) &= s(h_0(B))+s(h_0(B'))+\varepsilon\\
&= s(h_0(B))+s(h_0(B')) + 1 + \bar{s}(A)+\bar{s}(A') + \varepsilon \\
&= s(h_0(A))+s(h_0(A')) + \bar{s}(A)+\bar{s}(A') + (\bar{s}(A)+\bar{s}(A')) \varepsilon \\
&= s_\cC(A)+s_\cC(A') + (\bar{s}(A)+\bar{s}(A')) \varepsilon.
\end{align*}
So in all cases we get (\ref{eq:a=x}).

Now $s_x(B) = s_\cC(B)$, $s_x(B') = s_\cC(B')$, and $s_\cCC(P_c) = |c|_\cC$, and therefore
\begin{multline*}
t_I = s_\cC(B)+s_\cC(B')+|c|_\cC+s_\cC(A)+s_\cC(A')+(\bar{s}(A)+\bar{s}(A'))\varepsilon \\
+f_\cCC(I_B) +f_\cCC(I_{B'}) +f_\cCC(I_A) +f_\cCC(I_{A'}). 
\end{multline*}
As 
\begin{multline*}
f_\cCC(I_B) +f_\cCC(I_{B'}) +f_\cCC(I_A) +f_\cCC(I_{A'}) = s_\cC(B)(|c|_\cC+1)+s_\cC(B')(|c|_\cC+1) \\
+s_\cC(A)(|c|_\cC+1)+\bar{s}(A)\varepsilon+s_\cC(A')(|c|_\cC+1)+\bar{s}(A')\varepsilon
\end{multline*}
we get
\begin{multline*}
t_I = 2 (s_\cC(B)+s_\cC(B')+s_\cC(A)+s_\cC(A')+(\bar{s}(A)+\bar{s}(A'))\varepsilon) \\
+|c|_\cC (1+s_\cC(B)+s_\cC(B')+s_\cC(A)+s_\cC(A')) = 0.
\end{multline*}
If $b=x$ we have
\[
f_\cC(I) + f_x(I) = s_\cC(B)+\bar{s}(A)+(\bar{s}(B)+\bar{s}(A))\varepsilon.
\]
This is similar to the proof of (\ref{eq:a=x}), using a slightly different case-by-case analysis. Also notice that $s_x(B) = s_\cC(B)+1$ and $s_x(B') = s_\cC(B)$. Then
\begin{multline*}
t_I = s_\cC(B)+1+s_\cC(B')+|c|_\cC+ s_\cC(B)+\bar{s}(A)+(\bar{s}(B)+\bar{s}(A))\varepsilon+ \\
s_\cC(A')|c|_\cC+s_\cC(B')|c|_\cC+s_\cC(B') +s_\cC(A)|c|_\cC+1+\bar{s}(A)(1+\varepsilon)+s_\cC(B)|c|+\bar{s}(B)\varepsilon.
\end{multline*}
It is straightforward (and similar to the previous case) to check that $t_I=0$.

If $c=x$ we have, by another case-by-case analysis, that
\[
f_\cC(I) + f_x(I) = 1+(\bar{s}(B)+\bar{s}(B'))(1+\varepsilon).
\]
Now $s_x(B) = s_\cC(B)+1$ and $s_x(B') = s_\cC(B)+1$, and $s_\cCC(P_c) = |c|_\cC+1$, so that
\begin{multline*}
t_I = s_\cC(B)+s_\cC(B')+|c|_\cC+1+ 1+(\bar{s}(B)+\bar{s}(B'))(1+\varepsilon) \\
+ s_\cC(A)|c|_\cC+s_\cC(A')|c|_\cC+ s_\cC(B)|c|_\cC+s_\cC(B)+1+\bar{s}(B)(1+\varepsilon) \\
+ s_\cC(B')|c|_\cC+s_\cC(B')+1+\bar{s}(B')(1+\varepsilon).
\end{multline*}
Again it is easy to check that this adds up to $0\in \Z/2\Z$, and $\cCC$ is indeed framed. The result now follows from Lemma \ref{lm:funcsqconnect}.
\end{proof}

\begin{corollary}\label{cor:sqtwosigndep}
The second Steenrod square
\[
\Sq^2_{s,\varepsilon}\colon H^\ast(\cC_F;\Z/2\Z)\to H^{\ast+2}(\cC_F;\Z/2\Z)
\]
does not depend on the sign assignment $s$ for $\cC(n)$.
\end{corollary}

\begin{proof}
Let $u\in \Ob(\cn)$ and let $s' = s+\delta(u)$. The framed $1$-flow category $(\cC, s'_F,f'_F)$ is obtained from $(\cC,s_F,f_F)$ by sign changes in all $x\in F_u$. The result therefore follows from Lemma \ref{lm:changeofsign}.
\end{proof}
This shows the first part of Proposition \ref{prop:firstprop}. The dependence on $\varepsilon$ follows from the calculations in \ref{ss:sqcalcs}.

We can compose  $F\colon \cn \to \mathcal{B}_\sigma$ with the forgetful functor $\mathfrak{F}\colon \mathcal{B}_\sigma \to \mathcal{B}$. This results in another second Steenrod square $\Sq^2\colon H^\ast(\cC_{\mathfrak{F}F};\Z/2\Z)\to H^{\ast+2}(\cC_{\mathfrak{F}F};\Z/2\Z)$. Notice that this does not depend on $\varepsilon$, as $\varepsilon$ is never used in the construction of $\cC_{\mathfrak{F}F}$. 

Furthermore, since $\Z/2\Z$-coefficients ignore the sign assignment, we can naturally identify $H^\ast(\cC_{F};\Z/2\Z) = H^\ast(\cC_{\mathfrak{F}F};\Z/2\Z)$. We therefore have three second Steenrod squares on $H^\ast(\cC_{F};\Z/2\Z)$. As we will see, all three can be different.

\section{Odd Khovanov Homology}
Motivated by a spectral sequence with $\Z/2\Z$-coefficients from Khovanov homology of a link $L$ to the Heegaard Floer homology of the branched double-cover of $L$, Ozsv\'{a}th--Rasmussen--Szab\'{o} introduced Odd Khovanov homology in \cite{MR3071132}. We recall their definition, although our presentation is following \cite{MR4078823} more closely.

\subsection{A quick recollection of odd Khovanov homology}
Given a finite set $S$, we denote by $\Lambda^\ast(S)$ the exterior algebra over the free abelian group generated by $S$. This algebra has a natural grading that we call the $q$-grading. If $\Lambda^r(S)$ denotes the subgroup generated by words in $S$ of length $r$, we declare that its elements have $q$-degree $|S|-2r$.

Consider a link diagram $D$ with $n$ crossings that we assume to be ordered. Each crossing can be resolved as a $0$-smoothing and a $1$-smoothing such that the $1$-smoothing is obtained from the $0$-smoothing by a surgery. We assume that every crossing comes with an orientation of the surgery arc, and the convention of which smoothing is which is given in Figure \ref{fig:crossing_arrow}.

\begin{figure}[ht]
\begin{tikzpicture}[scale=0.6]
\crossing{0}{0}{1}{very thick}
\arroweast{0.3}{0.5}{0.4}
\node at (1.5,0.5) {:};
\smoothingup{2}{0}{1}{very thick}
\arroweast{3.3}{0.4}{1.4}
\smoothingup{3.7}{0.6}{0.6}{thick}
\arroweast{3.9}{0.9}{0.2}
\smoothinglr{5}{0}{1}{very thick}
\end{tikzpicture}
\caption{\label{fig:crossing_arrow}A crossing with $1$-smoothing on the right.}
\end{figure}
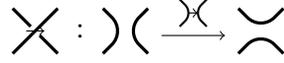 

For each $u\in \Ob(\cn)=\{0,1\}^n$ we can form the smoothing $\mathcal{S}_u$ by using the $u_i$-smoothing at the $i$-th crossing. Then $\mathcal{S}_u$ is a finite collection of circles, and we write $S_u$ for the components of $\mathcal{S}_u$.

We want to construct a functor $\mathfrak{K}_o\colon (\cn)^\op\to\Grd$ from the opposite category of $\cn$ to graded abelian groups. Since odd Khovanov homology is really cohomology, we use the opposite category of $\cn$. On objects we use $\mathfrak{K}_o(u) = \Lambda^\ast(S_u)$.

If $u\geq_1v$, consider the homomorphism $\psi_{v,u}\colon \Lambda^\ast(S_v)\to\Lambda^\ast(S_u)$ defined as follows. As $\mathcal{S}_u$ is obtained from $\mathcal{S}_v$ by a single surgery, we either have two components $s_1,s_2\in S_v$ which are merged to $s\in S_u$, or a component $s\in S_v$ splits into two components $s_1,s_2\in S_u$.

In the case of a merge we use the natural map $\Lambda^\ast(S_v)\to \Lambda^\ast(S_u)$ induced by the surjection $S_v\to S_u$ which sends both $s_1,s_2$ to $s$, and fixes all other components, for $\psi_{v,u}$. In the case of a split, we use the small arrow in Figure \ref{fig:crossing_arrow} to distinguish $s_1, s_2$. If we rotate this arrow counterclockwise by 90 degrees, it will point from one component to the other (in the case of Figure \ref{fig:crossing_arrow} it will point from the lower to the upper component in the $1$-smoothing). We call the component that is being pointed to $s_2$. We then define $\psi_{v,u}(x) = (s_1-s_2)\cdot \chi(x)$, where $\chi\colon \Lambda^\ast(S_v)\to \Lambda^\ast(S_u)$ is induced by the injection $S_v\to S_u$ sending $s$ to $s_1$ and fixing all other components.

For $u\geq_1 v_1,v_2 \geq_1 w$ with $v_1\not=v_2$ we get $\psi_{v_1,u}\circ \psi_{w,v_1} = \pm \psi_{v_2,u}\circ \psi_{w,v_2}$, depending on the local surgery picture. Figure \ref{fig:subcases} lists the possible cases. Let us write $\mathcal{S}_{u,w}$ for the local surgery picture arising this way.

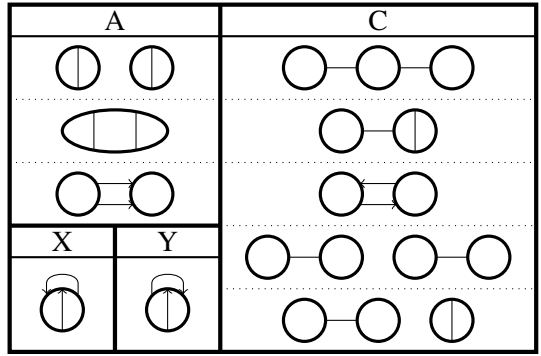
\begin{figure}[ht]
\begin{tikzpicture}[scale = 0.7]
\draw[ultra thick] (0,0.4) rectangle (10,7);
\draw[thick] (0,6.4) -- (10,6.4);
\draw[ultra thick] (4,0.4) -- (4,7);
\node at (2,6.7) {A};
\node at (1, 2.5) {X};
\node at (3, 2.5) {Y};
\node at (7, 6.7) {C};
\draw[very thick] (1.3,5.8) circle (0.4);
\draw (1.3, 5.4) -- (1.3, 6.2);
\draw[very thick] (2.7,5.8) circle (0.4);
\draw (2.7, 5.4) -- (2.7, 6.2);
\draw[very thick] (2, 4.6) ellipse (1 and 0.4);
\draw (1.6, 4.25) -- (1.6, 4.95);
\draw (2.4, 4.25) -- (2.4, 4.95);
\draw[very thick] (1.3,3.4) circle (0.4);
\draw[very thick] (2.7,3.4) circle (0.4);
\draw[->] (1.65, 3.2) -- (2.35, 3.2);
\draw[->] (1.65, 3.6) -- (2.35, 3.6);
\draw[very thick] (6.3,3.4) circle (0.4);
\draw[->] (6.65, 3.2) -- (7.35, 3.2);
\draw[<-] (6.65, 3.6) -- (7.35, 3.6);
\draw[very thick] (7.7,3.4) circle (0.4);
\draw[very thick] (6.3,4.6) circle (0.4);
\draw (6.7, 4.6) -- (7.3, 4.6);
\draw (7.7, 4.2) -- (7.7, 5);
\draw[very thick] (7.7,4.6) circle (0.4);
\draw[very thick] (7, 5.8) circle (0.4);
\draw (6,5.8) -- (6.6, 5.8);
\draw[very thick] (5.6, 5.8) circle (0.4);
\draw (7.4, 5.8) -- (8, 5.8);
\draw[very thick] (8.4, 5.8) circle (0.4);
\draw[very thick] (6.3,2.2) circle (0.4);
\draw (5.3, 2.2) -- (5.9, 2.2);
\draw[very thick] (7.7,2.2) circle (0.4);
\draw[very thick] (4.9, 2.2) circle (0.4);
\draw (8.1, 2.2) -- (8.7, 2.2);
\draw[very thick] (9.1, 2.2) circle (0.4);
\draw[very thick] (7, 1) circle (0.4);
\draw (6,1) -- (6.6, 1);
\draw[very thick] (5.6, 1) circle (0.4);
\draw[very thick] (8.4, 1) circle (0.4);
\draw (8.4, 0.6) -- (8.4, 1.4);
\draw[dotted] (0,5.2) -- (10, 5.2);
\draw[dotted] (0, 4) -- (10,4);
\draw[dotted] (4, 2.8) -- (10, 2.8);
\draw[dotted] (4, 1.6) -- (10, 1.6);
\draw[ultra thick] (0, 2.8) -- (4, 2.8);
\draw[ultra thick] (2, 2.8) -- (2, 0.4);
\draw[thick] (0, 2.2) -- (4, 2.2);
\draw[very thick] (1,1.2) circle (0.4);
\draw[->] (1.3, 1.5) -- (1.3, 1.7) to [out = 90, in = 90] (0.7,1.7) -- (0.7, 1.5);
\draw[very thick] (3,1.2) circle (0.4);
\draw[<-] (3.3, 1.5) -- (3.3, 1.7) to [out = 90, in = 90] (2.7,1.7) -- (2.7, 1.5);
\draw[->] (1, 0.8) -- (1, 1.6);
\draw[->] (3, 0.8) -- (3, 1.6);
\end{tikzpicture}
\caption{\label{fig:subcases}Commutation chart: Thick lines represent components in $S_{c}$, thin lines the surgery arcs. If a surgery arc has no orientation, then both orientations lead to the same result.}
\end{figure}

For the cases of type-A we get the two homomorphisms to be non-zero, and differ by a factor $-1$, for the cases of type-C, the two homomorphisms are equal and non-zero. In the remaining two types both homomorphisms are $0$.

\begin{definition}
A type-X {\em edge assignment} for the diagram $D$ with oriented crossings is a cochain $\varepsilon\in C^1(\cn;\Z/2\Z)$ such that for every $u\geq_2 w$ we have 
\[
\delta(\varepsilon)(C_{u,w}) = \left\{
\begin{array}{cc}
0 & \mathcal{S}_{u,w} \mbox{ is of type-C or type-Y}\\
1 & \mathcal{S}_{u,w} \mbox{ is of type-A or type-X}
\end{array}
\right.
\]
A type-Y edge assignment is defined similarly, but with the roles of X and Y reversed in the definition of $\delta(\varepsilon)$.
\end{definition}

Such edge assignments exist, see \cite[Lm.1.2]{MR3071132}. Our definition follows \cite{MR4078823} though, and to get an edge assignment as in \cite{MR3071132}, we have to add a sign assignment.

Using an edge assignment $\varepsilon$, we can now define for $u\geq_1 v$
\[
\mathfrak{K}_o(\phi_{u,v}^\op) = (-1)^{\varepsilon(C_{u,v})} \psi_{v,u},
\]
and extend this to a functor by composition.

Using a sign assignment $s$ for $\cn$, we now get a cochain complex $(C^\ast(D),\delta)$ by using
\[
C^k(D) = \bigoplus_{|u|=k}\mathfrak{K}_o(u),
\]
and $\delta = (-1)^{s(C_{u,v})}\mathfrak{K}_o(\phi^\op_{u,v})$ between the direct summands.

To make the coboundary $q$-grading preserving, and to make the chain homotopy type independent of the link diagram, we need to shift the gradings appropriately. Let $n_+$ be the number of positive crossings in the oriented link diagram $D$, and $n_-$ the number of negative crossings. We then define
\begin{equation}\label{eq:shifted}
CO^i(D) = C^{i+n_-}(D)\{i-2n_-+n_+\},
\end{equation}
where $\{j\}$ indicates a shift in $q$-grading, so that an element of $C^\ast(D)$ with $q$-degree $k$ has $q$-degree $k+j$ when viewed as an element of $C^\ast(D)\{j\}$.

By \cite[Thm.1.3]{MR3071132} the cohomology of $CO^\ast(D)$ is independent of the various choices. 

\subsection{Second Steenrod squares for odd Khovanov homology}\label{ss:squares}
For a link diagram $D$ we would like to get a framed $1$-flow category $\cC(D)$ with $C^\ast(\cC(D)) = CO^\ast(D)$. In \cite[\S 5.1]{MR4078823}, Sarkar--Scaduto--Stoffregen construct a functor $F\colon \cn \to \mathcal{B}_\sigma$ such that $\cC_F$ from Subsection \ref{sb:functorBurn} does exactly that. In order to get the second Steenrod square independent of all the choices involved, we need to recall the definition of $F$.

The value $F_u$ on objects $u\in \Ob(\cn)$ has to be a basis for the exterior algebra $\Lambda(S_u)$. To get such a basis, another choice is necessary. For each $u$ we choose an order $<$ on $S_u$. With this order, \cite{MR4078823} set
\[
F_u = \{ s_{i_1}\wedge \cdots \wedge s_{i_k} \in \Lambda(S_u)\mid k\in \{0,\ldots, |S_u|\}, s_{i_1}>\cdots>s_{i_k}\}  
\]
If $u\geq_1v$, we get for $y\in F_v$
\[
\mathfrak{K}(\phi^\op_{u,v})(y) = \sum_{x\in F_u} \varepsilon_{x,y} x
\]
for some $\varepsilon\in \{-1,0,1\}$. Then let
\[
F_{u,v} = \{(y,x)\in F_v\times F_u\mid \varepsilon_{x,y}\not=0\}.
\]
The source and target maps $s_{F_{u,v}}$, $t_{F_{u,v}}$ are the obvious projections; the sign is given by $\sigma_{F_{u,v}}(y,x) = \varepsilon_{x,y}$. Notice that $\sigma_{F_{u,v}}$ depends on the edge-assignment and on the orders on $S_u$ and $S_v$.

For $u\geq_1 v_1,v_2\geq_1 w$ there is a unique bijection between $F_{v_1,w}\circ F_{u,v_1}$ and $F_{v_2,w}\circ F_{u,v_2}$ which preserves source, target and sign, and it is shown in \cite[\S 5.1]{MR4078823} that this induces the required functor.  

Let us write $\cC_\varepsilon(D)$ for the framed 1-flow category obtained from $\cC_F$ by suspending so that $C^\ast(\cC_\varepsilon(D)) = CO^\ast(D)$. Here $\varepsilon\in \Z/2\Z$ indicates which $\varepsilon$ is used in the framing. 

By construction, $\cC_\varepsilon(D)$ is the disjoint union of framed $1$-flow categories $\cC_\varepsilon^j(D)$, which are the full subcategories generated by the objects that have $q$-degree $j$ when viewed as elements of $CO^\ast(D)$. In particular, $H^i(\cC^j_\varepsilon)$ is just the odd Khovanov homology of the link $L$ represented by $D$ in bidegree $(i,j)$. We denote this group by $\Kho{i,j}(L)$.

Using coefficients in $\Z/2\Z$ now gives rise to a second Steenrod square
\begin{equation}
\label{eq:steenrod}
\Sq^2_\varepsilon\colon \Kho{i,j}(L;\Z/2\Z) \to \Kho{i+2,j}(L;\Z/2\Z)
\end{equation}
but we need to check that it does not depend on the various choices made in the definition of $\cC_\varepsilon(D)$.

\begin{theorem}\label{thm:linkinv}
The second Steenrod square given by (\ref{eq:steenrod}) is a link invariant.
\end{theorem}

\begin{proof}
The proof is very similar to the proof of \cite[Thm.1.7]{MR4078823}, so we will not give too much detail. We note that by Subsection \ref{sb:functorBurn} the second Steenrod square does not depend on the sign-assignment nor the frame-assignment of the underlying cube $1$-flow category $\cC(n)$, nor on the $\delta\in \Z/2\Z$ used in the frame assignment.

Let us list the choices made coming from the odd Khovanov complex.
\begin{enumerate}
\item The ordering of the crossings in $D$.
\item The edge assignment.
\item The orientation at crossings.
\item The type of edge assignment.
\item The order of circles at each crossing.
\item The diagram $D$ for $L$.
\end{enumerate}
To see that it does not depend on the ordering of the crossings of the diagram, note that a change in the order is compensated by a change in sign- and frame-assignment of the cube $1$-flow category, and we have already seen that this does not affect the Steenrod square. 

If $\varepsilon$ and $\varepsilon'$ are edge assignments of the same type, then $\varepsilon+\varepsilon'\in C^1(\cn; \Z/2\Z)$ is a cocycle, hence a coboundary $\delta(x)$. But notice that $\delta(x)+s$ is another sign assignment for the cube, and the difference between edge assignments is compensated by changing the sign assignment for the cube. By Corollary \ref{cor:sqtwosigndep} the Steenrod square does not depend on the edge assignment.

By \cite[Lm.2.3]{MR3071132} the difference in a choice of orientations of crossings is compensated by a change of edge assignment of the same type. By the previous point changing orientations does not change the Steenrod square.

It follows from the proof of \cite[Lm.2.4]{MR3071132} that an edge assignment of type X is an edge assignment of type Y after a change of orientations. By the two previous points the Steenrod square does not depend on the type of edge assignment.

Changing the order of circles at a crossing changes the sign of signed correspondences in the functor $F\colon \cn\to \mathcal{B}_\sigma$. The effect on the framed $1$-flow category is described by sign changes at some of the elements of $F_u$. By repeated use of Lemma \ref{lm:changeofsign} the Steenrod square does not depend on this.

It remains to show that the Steenrod square does not depend on the diagram for the link. For this we need to consider the Reidemeister moves. The proof here is practically the same as the proof in \cite[\S 6]{MR3230817}. Instead of \cite[Lm.3.32]{MR3230817} we use Lemma \ref{lm:easyfunc} or Lemma \ref{lm:funcsqconnect}.
\end{proof}

Recall the forgetful functor $\mathfrak{F}\colon \mathcal{B}_\sigma\to\mathcal{B}$. Then $\mathfrak{F}\circ F\colon \cn\to\mathcal{B}$ recovers by \cite[Prop.5.3]{MR4078823} the functor used in \cite{MR4153651} to obtain the stable Khovanov homotopy type of \cite{MR3230817}. The corresponding (and suitably suspended) framed $1$-flow category $\cC_{\mathfrak{F}F}$ recovers the second Steenrod square of this homotopy type by \cite {MR4165986}. Let us denote this operation by
\[
\Sq^2\colon \Kh^{i,j}(L;\Z/2\Z) \to \Kh^{i+2,j}(L;\Z/2\Z).
\]
Note that $\Kh^{i,j}(L;\Z/2\Z) \cong \Kho{i,j}(L;\Z/2\Z)$, and we will drop the `o' from the notation in what follows.

\subsection{Reduced odd Khovanov homology} By picking a basepoint $p$ on a link diagram away from the crossings, a reduced version for odd Khovanov homology can be defined. Interestingly, the homology is independent of the basepoint, but for our Steenrod squares we only get this for knots.

Let $\cC_\varepsilon(D)$ be the framed $1$-flow category from the previous subsection. For every smoothing $\mathcal{S}_u$ let $s^p_u$ be the component containing the basepoint. Then let
\[
F_u^p = \{ w \in F_u\mid s^p_u\wedge w = 0 \}.
\]
These are the objects of $\cC_\varepsilon(D)$ which already contain $s^p_u$ in their word. We let $\widetilde{\cC}_\varepsilon(D)$ be the full subcategory spanned by the $F_u^p$. It is easy to see that this subcategory is upward closed. We also write $\widetilde{\cC}^{j+1}_\varepsilon(D) = \widetilde{\cC}_\varepsilon(D) \cap \cC^j_\varepsilon(D)$, and thus get a short exact sequence of framed $1$-flow categories
\[
0\longrightarrow \overline{\cC}^{j-1}_\varepsilon(D)\longrightarrow \cC^j_\varepsilon(D) \longrightarrow \widetilde{\cC}^{j+1}_\varepsilon(D) \longrightarrow 0
\]
where the complementary subcategory $\overline{\cC}^{j-1}_\varepsilon(D)$ is generated by those $w\in F_u$ with $s^p_u \wedge w \not=0$ which have $q$-degree $j$.

If we order each $S_u$ such that $s^p_u$ is the maximal element, we can identify the framed $1$-flow categories $\overline{\cC}^{j-1}_\varepsilon(D)$ and $\widetilde{\cC}^{j-1}(D)$ using multiplication by $s^p_u$ on objects.

\begin{proposition}
The second Steenrod square
\[
\widetilde{\Sq}^2_\varepsilon\colon \widetilde{\Kh}^{i,j}(L;\Z/2\Z)\to \widetilde{\Kh}^{i+2,j}(L;\Z/2\Z)
\]
arising from $\widetilde{\cC}^j_\varepsilon(D)$ does not depend on the various choices and where on a component the basepoint is placed.
\end{proposition}

\begin{proof}
This is completely analogous to the proof of Theorem \ref{thm:linkinv}. We note that the basepoint can move on a component by using Reidemeister moves on the link diagram in $S^2$ away from the basepoint by allowing moves going over $\infty$, compare \cite[\S 8]{MR3230817} and \cite[\S 5.4]{MR4078823}.
\end{proof}

The long exact sequence in odd Khovanov homology is known to split \cite{MR3071132}, but we will see that the second Steenrod square does not split.

\subsection{Concordance Invariants} Odd Khovanov homology does not appear to have an analogue of the Lee spectral sequence or the Rasmussen invariant. However, as in \cite[\S 5.6]{MR4078823} we can define a refinement of the $\Z/2\Z$ $s$-invariant which takes the second Steenrod square into account.

Let $D$ be a knot diagram and $\F$ a field. The {\em Bar-Natan complex} $\CBN^\ast(D;\F)$ is a deformation of the Khovanov complex $\CKh^{\ast,\ast}(D;\F)=C^\ast(\cC_{\mathfrak{F}F};\F)$ which admits a descending filtration $\mathcal{F}_j$ with $\CKh^{\ast,j}(D;\F) = \mathcal{F}_j/\mathcal{F}_{j+2}$, compare \cite{MR3189434}. The corresponding spectral sequence has exactly two non-trivial terms $E^{0,s\pm 1}_\infty$ for some $s\in 2\Z$, and this $s$ is the {\em Rasmussen $s$-invariant} $s_\F(K)$ of the knot $K$ with coefficients in the field $\F$, see \cite{MR2729272} and \cite{MR3189434}.

In \cite{MR3189434} Lipshitz--Sarkar describe a general technique to refine the Rasmussen-invariants using cohomology operations, and we recall their definition for any of the second Steenrod squares $\alpha\colon \Kh^{i,j}(K;\F_2) \to \Kh^{i+2,j}(K;\F_2)$, where $\F_2 = \Z/2\Z$ is the field with two elements. 

Denote $p\colon H^0(\mathcal{F}_j;\F_2) \to \Kh^{0,j}(K;\F_2) \cong H^0(\mathcal{F}_j/\mathcal{F}_{j+2};\F_2)$, and consider the following configurations
\begin{equation}\label{eq:fullconfigs}
\begin{tikzpicture}[baseline={([yshift=-.5ex]current bounding box.center)}]
\node at (0,0) {$\langle \tilde{a},\tilde{b}\rangle$};
\node at (3,0) {$\langle \hat{a}, \hat{b} \rangle$};
\node at (6,0) {$\langle a,b\rangle$};
\node at (9,0) {$\langle \bar{a},\bar{b}\rangle$};
\node at (0,1) {$\Kh^{-2,j}(K;\F_2)$};
\node at (3,1) {$\Kh^{0,j}(K;\F_2)$};
\node at (6,1) {$H^0(\mathcal{F}_j;\F_2)$};
\node at (9,1) {$\BN^0(K;\F_2)$};
\node at (0,2) {$\langle \tilde{a}\rangle$};
\node at (3,2) {$\langle \hat{a}\rangle$};
\node at (6,2) {$\langle a\rangle$};
\node at (9,2) {$\langle \bar{a}\rangle\not=0$};
\draw[right hook->] (0,0.25) -- (0,0.75);
\draw[right hook->] (3,0.25) -- (3,0.75);
\draw[right hook->] (6,0.25) -- (6,0.75);
\draw[-] (9.05,0.25) -- (9.05,0.75);
\draw[-] (8.95,0.25) -- (8.95,0.75);
\draw[right hook->] (0,1.75) -- (0,1.25);
\draw[right hook->] (3,1.75) -- (3,1.25);
\draw[right hook->] (6,1.75) -- (6,1.25);
\draw[right hook->] (9,1.75) -- (9,1.25);
\draw[->] (0.5,0) -- (2.5,0);
\draw[<-] (3.5,0) -- (5.5,0);
\draw[->] (6.5,0) -- (8.5,0);
\draw[->] (0.4,2) -- (2.6,2);
\draw[<-] (3.4,2) -- (5.6,2);
\draw[->] (6.4,2) -- (8.3,2);
\draw[->] (1.1,1) -- node[above] {$\alpha$} (2,1);
\draw[<-] (4,1) -- node[above] {$p$} (5.15,1);
\draw[->] (6.9,1) -- node[above] {$i^\ast$} (8.05,1);
\end{tikzpicture}
\end{equation}

\begin{definition}
Call an odd integer $j$ \em $\alpha$-half-full\em, if there exist $a\in H^0(\mathcal{F}_j;\F_2)$ and $\tilde{a}\in Kh^{-2,j}(K;\F_2)$ such that $p(a)=\alpha(\tilde{a})$, and such that $i^\ast(a)=\bar{a}\not=0$. That is, there exists a configuration as in the upper two rows of (\ref{eq:fullconfigs}).

Call an odd integer $j$ \em $\alpha$-full\em, if there exist $a,b\in H^0(\mathcal{F}_j;\F_2)$ and $\tilde{a},\tilde{b}\in Kh^{-2,q}(K;\F_2)$ such that $p(a)=\alpha(\tilde{a})$, $p(b)=\alpha(\tilde{b})$, and $i^\ast(a),i^\ast(b)$ generate $\BN^0(K;\F_2)$. That is, there exists a configuration as in the lower two rows of (\ref{eq:fullconfigs}).
\end{definition}

\begin{definition}
Let $K$ be a knot and $\alpha\colon \Kh^{i,j}(K;\F_2)\to \Kh^{i+2,j}(K;\F_2)$ be one of the second Steenrod squares $\Sq^2_0$, $\Sq^2_1$, or $\Sq^2$ described in Subsection \ref{ss:squares}. Then $r^\alpha_+,r^\alpha_-, s^\alpha_+, s^\alpha_-\in \Z$ are defined as follows.
\begin{align*}
r^\alpha_+(K) &= \max\{j\in 2\Z+1\,|\, j \mbox{ is $\alpha$-half-full}\}+1\\
s^\alpha_+(K) &= \max\{j\in 2\Z+1\,|\, j \mbox{ is $\alpha$-full}\}+3.
\end{align*}
If $\overline{K}$ denotes the mirror of $K$, we also set
\begin{align*}
r^\alpha_-(K) &= -r^\alpha_+(\overline{K})\\
s^\alpha_-(K) &= -s^\alpha_+(\overline{K}).
\end{align*}
We also write
\[
 s^\alpha(K) = (r^\alpha_+(K),s^\alpha_+(K),r^\alpha_-(K),s^\alpha_-(K)).
\]
\end{definition}

The proof in \cite[Lm.4.2]{MR3189434} carries over to show that 
\[
r^\alpha_+(K), s^\alpha_+(K)\in \{s_{\F_2}(K), s_{\F_2}(K)+2\}.
\]
 For $\alpha = \Sq^2$ it is shown in \cite[Thm.1]{MR3189434} that $|r^\alpha_\pm(K)|/2$ and $|s^\alpha_\pm(K)|/2$ are concordance invariants and lower bounds for the smooth slice genus $g_4(K)$. The proof carries over to $\alpha = \Sq^2_\varepsilon$ for $\varepsilon \in \Z/2\Z$.

\begin{theorem}
Let $\varepsilon\in \Z/2\Z$. Then $|r^{\Sq^2_\varepsilon}_\pm(K)|/2$ and $|s^{\Sq^2_\varepsilon}_\pm(K)|/2$ are concordance invariants and lower bounds for the smooth slice genus $g_4(K)$.
\end{theorem}

\begin{proof}
The proof is completely analogous to the proof of \cite[Thm.1]{MR3189434}, the only missing ingredient is that $\Sq^2_\varepsilon$ commutes with maps induced by link-cobordism. But this also follows as in \cite{MR3189434}. One writes the cobordism as a composition of Reidemeister and Morse moves, then observes that the induced maps on Khovanov homology can be realized by short exact sequences of $1$-flow categories, and the result follows from Lemma \ref{lm:easyfunc} and \ref{lm:funcsqconnect}.
\end{proof}

\section{Calculations}
The author's computer programme \verb+KnotJob+ has calculations of Steenrod squares implemented, and is available from the author's website. At the moment knots with 13 crossings or more can only be done with patience. It may be possible to deal directly with rational tangles as in \cite{MR4015262} to improve calculation times, but we have not further investigated this yet.

In order to get a non-trivial second Steenrod square, the Khovanov homology needs to support non-trivial groups in degree $i$ and $i+2$ for some $q$-degree. Among knots in the Rolfsen table this is satisfied by only a few knots. Furthermore, for such knots one can identify the stable homotopy type from knowledge of the second Steenrod square and its interaction with Bockstein homomorphisms, compare \cite[\S 4]{MR3252965}. We will refer to these spaces, even if we only have conjecturally an identification of the stable homotopy type in the odd case.

\subsection{Steenrod Squares}\label{ss:sqcalcs}
In Table \ref{fig:calcs} we list any non-trivial second Steenrod squares among the prime knots with up to $10$ crossings. The information is given in the form of words for Chang spaces, see \cite{MR1361886}. An $\eta$ means that there is a non-trivial second Steenrod square $\Sq^2(x)$, a subscript $2$ means that the first Steenrod square $\Sq^1(x)$ is also non-zero. Finally, $\eta 2$ means that $0\not=\Sq^2(x) = \Sq^1(y)$ for some $x,y$.

Note that we consider all potential non-trivial second Steenrod squares, that is, $q$-degrees where the width of the cohomology is $3$. An empty entry means that the second Steenrod square is $0$. One can detect several patterns, although we have not investigated this further for more crossings. We notice that $\Sq^2_0$ only uses the Chang words $\eta2$ and $_2\eta2$, while $\Sq^2_1$ only uses $_2\eta$ and $_2\eta2$. Also, we have $\Sq^2$ non-trivial if and only if $\Sq^2_0$ or $\Sq^2_1$ is non-trivial. However, we do not expect this behaviour to be universal.

\begin{table}[ht]
\begin{tabular}{|c||c||c|c|c||c||c||c|c|c|}
\hline
Knot & $q$ & $\Sq^2$ & $\Sq^2_0$ & $\Sq^2_1$ & Knot & $q$ & $\Sq^2$ & $\Sq^2_0$ & $\Sq^2_1$\\
\hline
\hline
$8_{19}$ &  $11$ & $_2\eta$ & $\eta2$ &   &
$10_{152}$ & $13$ & $_2\eta$ & $\eta2$ & \\
\cline{1-5}
$9_{42}$ & $1$ & $\eta2$ & & $_2\eta$ &
& $15$ & & & \\
\cline{1-5}
$10_{124}$ & $13$ & $_2\eta$ & $\eta2$ & & 
& $19$ & $\eta2$ & & $_2\eta$ \\
\cline{6-10}
 & $19$ & $\eta2$ & & $_2\eta$ &
 $10_{153}$ & $-5$ & $_2\eta$ & $\eta2$ & \\
\cline{1-5}
$10_{128}$ & $11$ & $_2\eta$ & $\eta2$ & &
& $-3$ & $\eta2$ & $\eta2$ & $_2\eta2$  \\
\cline{1-5}
$10_{132}$ & $-9$ & $_2\eta$ & $\eta2$ & &
& $-1$ & & & \\
& $-7$ & $\eta2$ & $\eta2$ & $_2\eta2$ &
& $1$ & $_2\eta$ & $_2\eta2$ & $_2\eta$ \\
& $-3$ & $_2\eta$ & $\eta2$ & &
& $3$ & $\eta2$ & & $_2\eta$ \\
\hline
$10_{136}$ & $1$ & $\eta2$ & & $_2\eta$ &
$10_{154}$ & $11$ & $_2\eta$ & $\eta2$ & \\
\cline{1-5}
$10_{139}$ & $13$ & $_2\eta$ & $\eta2$ & &
& $13$ & & & \\
& $15$ & & & &
& $17$ & $\eta2$ & & $_2\eta$ \\
\cline{6-10}
& $19$ & $\eta2$ & & $_2\eta$ &
$10_{161}$ & $11$ & $_2\eta$  & $\eta_2$ & \\
\cline{1-5}
$10_{145}$ & $-15$ & $_2\eta$ & $\eta2$ & &
& $13$ & & & \\
& $-13$ & $\eta2$ & $\eta2$ & $_2\eta2$ &
& $17$ & $\eta2$ & & $_2\eta$ \\
\cline{6-10}
& $-11$ & & & &
\multicolumn{5}{c|}{ } \\
& $-9$ & $_2\eta2$ & & $_2\eta$ & 
\multicolumn{5}{c|}{ } \\
\hline
\end{tabular}
\vspace{0.2cm}
\caption{\label{fig:calcs}The second Steenrod squares for knots in the Rolfsen table.}
\end{table}

Indeed, for the knot $13^n_{3663}$ we get an $\eta$ word in $q$-degree $-5$ for $\Sq^2_0$. We note that in $q$-degree $1$ we get a Baues--Hennes word $_2\eta^2\xi$ for $\Sq^2_0$. Here $\xi$ stands for another non-trivial $\Sq^2_0$, but in a higher homological degree. At the time of writing we do not know a knot which admits an $\eta$ word for $\Sq^2$, compare \cite[Qn.5.2]{MR3252965}. 

The short exact sequence
\[
0\longrightarrow \widetilde{\Kh}^{i,j+1}(L;\Z/2\Z)\longrightarrow \Kh^{i,j}(L;\Z/2\Z)\stackrel{p}{\longrightarrow} \widetilde{\Kh}^{i,j-1}(L;\Z/2\Z)\longrightarrow 0
\]
is known to split, but this splitting does not extend to the second Steenrod squares. Indeed, all $\widetilde{\Sq}^2_\varepsilon$ are zero for the Rolfsen table by degree reasons. At the moment \verb+KnotJob+ does not support computations of reduced Steenrod squares, but we can derive a non-triviality result from the short exact sequence for the knot $K=13^n_{3663}$.

In $q$-degree $3$ we get a word $_2\eta2$ for $\Sq^2_0$ starting in homological degree $0$. Calculations show that
\[
\Kh^{i,3}(K;\Z/2\Z) = \widetilde{\Kh}^{i,2}(K;\Z/2\Z) \cong \Z/2\Z
\]
for $i = 0,1$. Also, if $x\in \Kh^{0,3}(K;\Z/2\Z)$ is the element with $\Sq^2_0(x)$ non-zero, we have $\Sq^2_0(x) = \Sq^1(y)$ for $y\in \Kh^{1,3}(K;\Z/2\Z) = \widetilde{\Kh}^{1,2}(K;\Z/2\Z)$. Integral calculations show that $\widetilde{\Sq}^1(y) \not= 0$ in $\widetilde{\Kh}^{2,2}(K;\Z/2\Z)$, so 
\[
\widetilde{\Sq}^2_0(x) = p\circ \Sq^2_0(x) = p\circ \Sq^1(y) = \widetilde{\Sq^1}(y)\not=0
\]
by naturality.

To see that the second Steenrod square $\widetilde{\Sq}^2_\varepsilon$ can depend on the component of the basepoint, consider the split union $L = K \sqcup O$, where $K$ is any of the knots in Table \ref{fig:calcs}. If we place the basepoint on the unknot component, $\widetilde{\Sq}^2_\varepsilon$ agrees with $\Sq^2_\varepsilon$ for $K$. If we place the basepoint on $K$, then $\widetilde{Sq}^2_\varepsilon$ is trivial.

\subsection{Split unions}
Given two links $L_1$, $L_2$, denote the split union by $L_1\sqcup L_2$. The odd Khovanov homology of the split union behaves slightly different from the even version, compare \cite{MR3475078}, but with $\Z/2\Z$-coefficients we can still express it via a tensor product of the homologies of $L_1$ and $L_2$. 

For the stable homotopy type of \cite{MR3230817} one gets indeed a nice formula for the split union in terms of the smash product of the stable homotopy types for $L_1$ and $L_2$, \cite[Thm.1]{MR4153651}. In particular, this was used to show that the stable homotopy type of $T(2,3)\sqcup T(2,3)$ has a non-trivial second Steenrod square, compare the proof of \cite[Cor.1.4]{MR4153651}.

For the odd Khovanov homology of $T(2,3)\sqcup T(2,3)$ calculations show that we have
\[
\Kho{i,j}(T(2,3)\sqcup T(2,3);\Z) \cong \bigoplus_{\stackrel{j_1+j_2=j}{i_1+i_2=i}}\Kho{i_1,j_1}(T(2,3);\Z)\otimes \Kho{i_2,j_2}(T(2,3);\Z),
\]
with the lack of torsion meaning that we do not get any Tor-terms.

Any stable homotopy type for odd Khovanov homology for $T(2,3)$ has to be a wedge of spheres for homological width reasons, and in particular the smash product of two such copies is a wedge of spheres. The second Steenrod square in this smash product is therefore trivial.

Interestingly, in $q$-degree $14$ we get
\begin{align*}
\Kho{4,14}(T(2,3)\sqcup T(2,3);\Z) &\cong \Z \cong \Kho{6,14}(T(2,3)\sqcup T(2,3);\Z), \\ 
\Kho{5,14}(T(2,3)\sqcup T(2,3);\Z) &\cong \Z^4,
\end{align*}
and computation shows that both
\[
\Sq^2_0,\Sq^2_1\colon \Kh^{4,14}(T(2,3)\sqcup T(2,3);\Z/2\Z)\to \Kh^{6,14}(T(2,3)\sqcup T(2,3);\Z/2\Z)
\]
are isomorphisms.

If our Steenrod squares are the second Steenrod squares of a stable homotopy type for odd Khovanov homology, we now see that this stable homotopy type for $T(2,3)\sqcup T(2,3)$ could not be the smash product of the stable homotopy types for $T(2,3)$. Indeed, because of the absence of torsion we would get the Chang space with word $\eta$ in this case together with four spheres.

We note that in $q$-degree $14$ we also get a non-trivial second Steenrod square in the Lipshitz-Sarkar stable homotopy type for $T(2,3)\sqcup T(2,3)$. In this case the Chang space has word $_2\eta2$.

\subsection{Spanier--Whitehead duality}
If $\overline{L}$ denotes the mirror of a link $L$, we get the isomorphism
\[
\Kh^{i,j}(L;\Z/2\Z) \cong \Hom(\Kh^{-i,-j}(\overline{L};\Z/2\Z), \Z/2\Z)
\]
by turning the cochain complex of a diagram `upside down' and using universal coefficients. In the case of the Lipshitz--Sarkar stable homotopy type one gets a stronger statement involving the Spanier--Whitehead dual of the mirror \cite{MR4153651}. Its effect on cohomology groups is that the diagram
\[
\begin{tikzpicture}
\node at (0,2) {$\Kh^{i,j}(L;\Z/2\Z)$};
\node at (5.6,2) {$\Hom(\Kh^{-i,-j}(\overline{L};\Z/2\Z), \Z/2\Z)$};
\node at (0,0) {$\Kh^{i+2,j}(L;\Z/2\Z)$};
\node at (5.6,0) {$\Hom(\Kh^{-i-2,-j}(\overline{L};\Z/2\Z), \Z/2\Z)$};
\draw[->] (1.35,2) -- node[above] {$\cong$} (3,2);
\draw[->] (1.5,0) -- node[above] {$\cong$} (2.8,0);
\draw[->] (0,1.7) -- node[right] {$\Sq^2$} (0,0.3);
\draw[->] (5.6,1.7) -- node[right] {$(\Sq^2)^\ast$} (5.6,0.3);
\end{tikzpicture}
\]
commutes.

Calculations on mirrors show that this diagram does not commute if we use the second Steenrod square $\Sq^2_\varepsilon$ on the vertical arrows. Rather it seems that we should combine $\Sq^2_\varepsilon$ with the dual of $\Sq^2_{1+\varepsilon}$. In the case of the Rolfsen table we get this confirmed by calculations. In general we conjecture this behaviour.

\begin{conjecture}
Let $L$ be a link and $\overline{L}$ its mirror. Then for $\varepsilon\in \Z/2\Z$ the diagram
\[
\begin{tikzpicture}
\node at (0,2) {$\Kh^{i,j}(L;\Z/2\Z)$};
\node at (5.6,2) {$\Hom(\Kh^{-i,-j}(\overline{L};\Z/2\Z), \Z/2\Z)$};
\node at (0,0) {$\Kh^{i+2,j}(L;\Z/2\Z)$};
\node at (5.6,0) {$\Hom(\Kh^{-i-2,-j}(\overline{L};\Z/2\Z), \Z/2\Z)$};
\draw[->] (1.35,2) -- node[above] {$\cong$} (3,2);
\draw[->] (1.5,0) -- node[above] {$\cong$} (2.8,0);
\draw[->] (0,1.7) -- node[right] {$\Sq^2_\varepsilon$} (0,0.3);
\draw[->] (5.6,1.7) -- node[right] {$(\Sq^2_{1+\varepsilon})^\ast$} (5.6,0.3);
\end{tikzpicture}
\]
commutes.
\end{conjecture}

\subsection{Concordance invariants}

For width reasons, the invariants $r_\pm^{\Sq^2_\varepsilon}(K)$ are going to agree with $s_{\F_2}(K)$ for knots $K$ with a small number of crossings. To see that $s_\pm^{\Sq^2_\varepsilon}(K)$ can differ from $s_{\F_2}(K)$ let $K$ be the mirror of $9_{42}$. Calculations show that $\Kh^{0,-1}(K;\F_2)\cong \F_2$, and $\Sq^2_0$ surjects onto this group. It is now straightforward to check that $-1$ is $\Sq^2_0$-full, resulting in $s_+^{\Sq^2_\varepsilon}(K)=2>s_{\F_2}(K)=0$, thus showing Theorem \ref{thm:secthm}.

In general we cannot simply read off the refinements from a calculation of the second Steenrod square, but \verb+KnotJob+ is also able to calculate $s^\alpha(K)$ for $\alpha$ one of the second Steenrod squares. Interestingly, for all prime knots $K$ with up to $13$ crossings we get $s^{\Sq^2}(K) =s^{\Sq^2_0}(K)$, while $s^{\Sq^2_1}(K)$ is always constant. Overall, there are $324$ prime knots with up to $13$ crossings for which $s^{\Sq^2}(K)$ is not constant. 

We note that there is also a refinement $s^{\Sq^1_o}(K)$ corresponding to the first Steenrod square for odd Khovanov homology, see \cite{MR4078823}, and which can be calculated by \verb+KnotJob+. Interestingly, for all the prime knots with up to $14$ crossing and non-constant $s^{\Sq^2}(K)$  we also have $s^{\Sq^1_o}(K)$ non-constant\footnote{The converse does not hold: $11^n_{38}$ and $12^n_{25}$ have non-constant $s^{\Sq^1_o}(K)$, but constant $s^{\Sq^2}(K)$.}. But $s^{\Sq^1_o}_\pm(K) = s^{\Sq^2}_\pm(K) \not= s_\F(K)$ suggests the existence of $\Sq^2(x) = \Sq^1_o(y)\not=0$, which survives the spectral sequence. But we expect no $y'$ with $\Sq^2(x) = \Sq^1_e(y')$, as this would imply $s^{\Sq^1_e}_\pm(K) \not= s_\F(K)$ which is very rare and does not happen for knots with less than $14$ crossings. For the two knots with crossing number $14$ and $s^{\Sq^1_e}_\pm(K) \not= s_\F(K)$ both $s^{\Sq^2}(K)$ and $s^{\Sq^1_o}(K)$ are constant.


The knot $K=15^n_{41127}$ has non-constant $s^{\Sq^2}(K)$, but $s^{\Sq^1_o}(K)$ is constant. We also get $s^{\Sq^2_0}(K)$ is constant, while $s^{\Sq^2_1}(K)$ is non-constant. At the moment, this is the only knot we know of for which $s^{\Sq^2_1}(K)$ is non-constant. However, we have made very few calculations for knots with more than $13$ crossings, so there may well be others among knots with crossing number $14$ or $15$. This is also the only knot we know where $s^{\Sq^2}(K)$ and $s^{\Sq^2_0}(K)$ are different. Again, this may be due to lack of calculations, but it raises the

\begin{question}
Is $s^{\Sq^2}(K)$ non-constant if and only if $s^{\Sq^2_0}(K)$ is non-constant or $s^{\Sq^2_1}(K)$ is non-constant?
\end{question}



\appendix
\section{Framing the cube $1$-flow category}
Our framed $1$-flow categories very much depend on (\ref{eq:framechange}), which tells us how to change a framing after a change of signs. Given that different values of $\varepsilon$ can lead to different results for signed covers of the cube $1$-flow category, it is only natural to ask whether there exist different ways to change the framing of $\cC(n)$ after a change of signs.

So assume that $s$ is a sign assignment for $\cC(n)$ and $f$ a corresponding frame assignment. For $x\in \Ob(\cn)$ we get a new sign assignment $s' = s +\delta(x^\ast)$. Let us define a new pre-framing $f'$ by
\[
f'(C_{u,w}) = \left\{
\begin{array}{cc}
f(C_{u,w}) + 1 & w=x \\
f(C_{u,w}) & w\not=x
\end{array}
\right.
\]
This means we only change the framing value if the signs of the lower two edges changes.

The same argument used in the proof of Lemma \ref{lm:frame4signchange} shows that $f'$ is indeed a framing for $(\cC(n),s')$, and the formula is a lot simpler than (\ref{eq:framechange}).

The problem is that this formula depends on the coboundary, and we do not get the nice change simply by looking at the changes in sign on a cube $C_{u,w}$. For example, assume that $u\geq_1v_1,v_2\geq_1w$. Then $s+\delta(w^\ast+v_1^\ast)$ and $s+\delta(v_2^\ast+u^\ast)$ lead to the same sign changes on $C_{u,w}$, but result in different framings. It is not clear to us how to resolve this problem similar to the {\em increasing homological grading} formula used in (\ref{eq:framechange}), in order to get a good change-of-framing formula for signed covers of $\cC(n)$.

Nevertheless, there may be other ways to get a working formula to frame signed covers. In particular, it could involve multiplying signs, or the degree of objects $u\in \Ob(\cn)$. If one only wants to allow linear combinations in the signs, and insist on the first two lines of (\ref{eq:framechange}), the remaining lines are determined. If we are willing to change line two, there is another formula that can be used. This can also involve $\delta\in \Z/2\Z$, but Lemma \ref{lm:nodeltaissue} also applies for it.

We use the same $a,b,c,d\in\Z/2\Z$ as in (\ref{eq:framechange}). Assume $(\cC(n),s,f)$ is framed, and let $s'$ be another sign assignment for the cube. Define
\begin{equation}\label{eq:altframechange}
f''_{\varepsilon}(C_{u,w}) = 
\left\{
\begin{array}{cc}
f(C_{u,w}) & \mbox{if } a=a', b=b', c=c', d=d' \\[0.2cm]
f(C_{u,w}) + c+d & \mbox{if } a\not=a', b\not=b', c=c', d=d' \\[0.2cm]
f(C_{u,w})& \mbox{if } a=a', b=b', c\not=c', d\not=d' \\[0.2cm]
f(C_{u,w}) + d & \mbox{if } a\not=a', b=b', c=c', d\not=d' \\[0.2cm]
f(C_{u,w}) + c & \mbox{if } a=a', b\not=b', c\not=c', d=d' \\[0.2cm]
f(C_{u,w}) + \varepsilon + d & \mbox{if } a\not=a', b=b', c\not=c', d=d' \\[0.2cm]
f(C_{u,w}) + \varepsilon + c & \mbox{if } a=a', b\not=b', c=c', d\not=d' \\[0.2cm]
f(C_{u,w}) + c + d & \mbox{if } a\not=a', b\not=b', c\not=c', d\not=d' 
\end{array}
\right.
\end{equation}

The proof of Lemma \ref{lm:frame4signchange} carries over to show that $(\cC(n),s',f'')$ is framed.

Given a framing $(\cC(n),s,f)$ and a $(0,\varepsilon)$-signed cover $(\cC,s_\cC,f_\cC^\varepsilon)$, we can form a new framing $\tilde{f}_\cC^\varepsilon$ by using (\ref{eq:altframechange}) instead of (\ref{eq:framechange}). We get the same Steenrod square after a small adjustment in the $\varepsilon$ value.

\begin{lemma}
The framed $1$-flow category $(\cC,s_\cC,f_\cC^\varepsilon)$ admits the same second Steenrod square as the framed $1$-flow category $(\cC,s_\cC,\tilde{f}_\cC^{1+\varepsilon})$.
\end{lemma}

\begin{proof}
We need to choose an appropriate framing on the mapping cone $\cCC$. For the signs we simply use $s_\cC$ on both subcategories, and $s_\cCC(P_a)=|a|_\cC$ for any $a\in \Ob(\cC)$.

Before we choose the framing, observe that
\[
f'_\varepsilon(C_{u,w}) + f''_{1+\varepsilon}(C_{u,w}) = a(a+a')+b(b+b')+c(c+c')+d(d+d').
\]
We therefore choose
\[
f_\cCC(I_A) = |a|_\cC s_\cC(A) + s(h_0(A)) (s(h_0(A))+s_\cC(A)).
\]
It is straightforward to show that this gives rise to a framed $1$-flow category, which then proves the lemma by Lemma \ref{lm:funcsqconnect}.
\end{proof}

This shows that only allowing linear combinations in signs and keeping the first line of (\ref{eq:framechange}) cannot produce another second Steenrod square.

\begin{remark}
We could have used $1+\varepsilon$ in (\ref{eq:altframechange}) instead of $\varepsilon$ to get a cleaner statement. However, as in (\ref{eq:framechange}) we get that the sixth line is obtained by using the second and fifth line, or the third and fourth line. Choosing $\varepsilon = 0$ then corresponds to using the {\em increasing homological grading} convention. It does seem to be surprising that there is a change in $\varepsilon$ when moving from (\ref{eq:framechange}) to (\ref{eq:altframechange}).
\end{remark}

\bibliography{KnotHomology}
\bibliographystyle{amsalpha}

\end{document}